\documentclass{compositio}
\usepackage{amssymb,bbm}
\usepackage{amscd,amsfonts,amsmath,amssymb,bbm}
\usepackage{enumerate,epsf,fancyhdr,float,graphicx}
\usepackage{latexsym,mathrsfs,multirow}
\usepackage{wasysym}
\usepackage{xypic}
\usepackage[all]{xy}\usepackage[OT2,T1]{fontenc}
\usepackage[modulo,mathlines,displaymath, running]{lineno}
\usepackage{amsmath,mathrsfs,enumerate,wasysym}
\usepackage{xypic,hhline}
\usepackage[all]{xy}
\usepackage[OT2,T1]{fontenc}
\usepackage[modulo,mathlines,displaymath]{lineno}
\DeclareSymbolFont{cyrletters}{OT2}{wncyr}{m}{n}
\DeclareMathSymbol{\Sha}{\mathalpha}{cyrletters}{"58}
\usepackage[modulo,mathlines,displaymath]{lineno}

\DeclareSymbolFont{cyrletters}{OT2}{wncyr}{m}{n}\DeclareMathSymbol{\Sha}{\mathalpha}{cyrletters}{"58}

\newcommand{\Q}{\mathbb{Q}}\newcommand{\C}{\mathbb{C}}\newcommand{\Z}{\mathbb{Z}}\newcommand{\N}{\mathbb{N}}\newcommand{\R}{\mathbb{R}}
\newcommand{\links}{\left(\begin{array}{cc}}\newcommand{\rechts}{\end{array}\right)}\newcommand{\bai}{\left[\begin{array}{cc}}\newcommand{\dai}{\end{array}\right]}\newcommand{\hidari}{\left(\begin{array}{c}}\newcommand{\migi}{\end{array}\right)}\newcommand{\roots}{\smat{-1 & -1\\  \beta & \alpha}}\newcommand{\CC}{{\mathcal A}}\newcommand{\CCC}{{\mathcal C}}\newcommand{\Tam}{{\mathrm{Tam}}}\newcommand{\Gal}{{\mathrm{Gal}}}\newcommand{\val}{{\mathrm{val}}}\newcommand{\rank}{{\mathrm{rank}}}\newcommand{\ord}{{\mathrm{ord}}}\renewcommand{\O}{{\mathcal{O}}}\newcommand{\sign}{{\mathrm{sign}}}\newcommand{\gM}{{\mathfrak M}}\renewcommand{\geq}{\geqslant}\renewcommand{\leq}{\leqslant} \newcommand{\smat}[1]{\left( \begin{smallmatrix} #1 \end{smallmatrix} \right)}\newcommand{\mat}[1]{\left[ \begin{smallmatrix} #1 \end{smallmatrix} \right]}\renewcommand{\H}{{\mathcal{H}}}\renewcommand{\phi}{{\varphi}}\newcommand{\G}{{\mathcal G}}\newcommand{\Log}{\mathcal Log}\newcommand{\Loghat}{\widehat{\Log}_{\alpha,\beta}(1+T)}\newcommand{\Lvec}{\widehat{\overrightarrow{L}}_p(f,\omega^i,T)}\newcommand{\Logarithm}{\Log_{\alpha,\beta}(1+T)}

\renewcommand{\contentsname}{Contents\\{\footnotesize\normalfont(A table
of contents should normally not be included)}}

\newtheorem{theorem}{Theorem}[section]

\newtheorem{conjecture}[theorem]{Conjecture}

\newtheorem{corollary}[theorem]{Corollary}

\newtheorem{example}[theorem]{Example}
\newtheorem{lemma}[theorem]{Lemma}

\newtheorem{proposition}[theorem]{Proposition}
\newtheorem{open problem}[theorem]{Open Problem}

\newtheorem{observation}[theorem]{Observation}
\newtheorem{definition}[theorem]{Definition}

\newtheorem{convention}[theorem]{Convention}
\newtheorem{notation}[theorem]{Notation}

\newtheorem{remark}[theorem]{Remark}
\begin{document}

\title{On pairs of $p$-adic $L$-functions for weight two modular forms}
\author{Florian Sprung}
\email{fsprung@math.princeton.edu}
\address{Institute for Advanced Study, School of Mathematics, 1 Einstein Dr, Princeton, NJ 08540, USA}

\dedication{Dedicated to Barry, Joe, and Rob}
\classification{Primary: 11G40, 11F67. Secondary: 11R23, 11G05}
\keywords{Birch and Swinnerton-Dyer, $p$-adic $L$-function, elliptic curve, modular form, \v{S}afarevi\v{c}-Tate group, Iwasawa Theory}
\thanks{This material is based upon work supported by the National Science Foundation under agreement No. DMS-1128155. Any opinions, findings and conclusions or recommendations expressed in this material are those of the author and do not necessarily reflect the views of the National Science Foundation.}

\begin{abstract}
 The point of this paper is to give an explicit $p$-adic analytic construction of two Iwasawa functions $L_p^\sharp(f,T)$ and $L_p^\flat(f,T)$ for a weight two modular form $\sum a_nq^n$ and a good prime $p$. This generalizes work of Pollack who worked in the supersingular case and also assumed $a_p=0$. The Iwasawa functions work in tandem to shed some light on the Birch and Swinnerton-Dyer conjectures in the cyclotomic direction: We bound the rank and estimate the growth of the Tate-Shafarevich group in the cyclotomic direction analytically, encountering a new phenomenon for small slopes.
\end{abstract}

\maketitle

%\linenumbers

%\vspace*{6pt}\tableofcontents  % for this guide only.
% A table of contents should normally not be included
\section{Introduction}
\label{sec:introduction}

Let $f=\sum a_nq^n$ be a weight two modular form. The idea of attaching a $p$-adic $L$-function to $f$ goes back to at least Mazur and Swinnerton-Dyer in the case where the associated abelian variety $A_f$ is an elliptic curve. They analytically constructed a power series $L_\alpha(f,T)$, whose behavior at special values $\zeta_{p^n}-1$ corresponding to finite layers $\Q_n$ of the cyclotomic $\Z_p$-extension $\Q_\infty$ of $\Q$ should mirror that of the rational points $A_f(\Q_n)$ and the \v{S}afarevi\v{c}-Tate group $\Sha(A_f/\Q_n)$ in view of the Birch and Swinnerton-Dyer conjectures. Identifying algebraic numbers with $p$-adic numbers via a fixed embedding $\overline{\Q}\hookrightarrow\C_p$ , we may give $a_p$ a valuation $v$. Their crucial assumption was that $p$ be good and $v=0$ ($p$ is \textit{ordinary}), so that $L_\alpha(f,T)$ is an Iwasawa function, i.e. analytic on the closed unit disc. Since $L_\alpha(f,T)$ is non-zero by work of Rohrlich \cite{rohrlich}, we can extract Iwasawa invariants, responsible for that behavior of $L_\alpha(f,T)$ which under the Main Conjecture corresponds to a bound for $\rank(A_f(\Q_\infty))$ and a description of the size of $\Sha(A_f/\Q_n)[p^\infty]$. Skinner and Urban \cite{skinnerurban} settled the Main Conjecture in many cases. 

The construction of $L_\alpha(f,T)$ has been generalized to the supersingular (i.e. $v>0$) case as well \cite{amicevelu,vishik}, in which there are two power series $L_\alpha(f,T)$ and $L_\beta(f,T)$. They are not Iwasawa functions, and thus not amenable for estimates for $\rank(A_f(\Q_\infty))$ or $\Sha(A_f/\Q_n)$ directly. Nevertheless, the results of Rohrlich show that $L_\alpha(f,T)$ and $L_\beta(f,T)$ vanish at finitely many special values $\zeta_{p^n}-1$, so that the analytic rank of $A_f(\Q_\infty)$ is bounded. His results are effective\footnote{In the elliptic curve case, $L_\alpha(\zeta_{p^n}-1)\neq 0$ for $p^{n}>\max(10^{1000}N^{170},10^{120p}p^{7p}, 10^{6000}p^{420})$, where $N$ is the conductor of $E=A_f$ \cite[p. 416]{rohrlich}.}. 
%In the case $a_p=0$, Pollack introduced a pair of Iwasawa functions, whose invariants he then used (in the elliptic modular case) to bound the analytic rank of $A_f(\Q_\infty)$ assuming further ${p\equiv 3 \pmod 4}$, and to give estimates for the \textit{analytic size} for $\Sha(A_f/\Q_n)[p^\infty]$, i.e. estimates for the special values $L_\alpha(\zeta_{p^n}-1)$ and $L_\beta(\zeta_{p^n}-1)$.

The main theorem of Part $1$ in this paper obtains a \textit{pair} of appropriate Iwasawa functions in the general good reduction case so that $p$ can be ordinary or supersingular. They are unique when $p$ is supersingular, and generalize the results of Pollack in the $a_p=0$ case. Note that the hypothesis $a_p=0$ is very restrictive since the vast majority of supersingular modular abelian varieties have modular forms failing this condition. The philosophy of using \textit{pairs} of objects has its origins in the work of Perrin-Riou \cite{PR,PR93} in the supersingular case in which the pair consisting of $L_\alpha(f,T)$ and $L_\beta(f,T)$ is considered as one object as a power series with coefficients in the Dieudonn\'{e} module. Our main theorem generalizes the construction of a pair of functions in the case of elliptic curves and supersingular primes \cite{shuron} via Kato's zeta element. As pointed out in \cite[Remark 5.26]{llz}, the methods in \cite{shuron} extend to the case $v\geq \frac{1}{2}$ as well. In this paper, we have completely isolated the analytic aspects of the theory and are thus able to treat the much harder case $v<\frac{1}{2}$. In the supersingular case, we prove a functional equation for this pair, which corrects a corresponding statement in \cite{pollack} when reduced to the $a_p=0$ case. Also, we give a quick proof that $L_\alpha(f,T)$ and $L_\beta(f,T)$ have finitely many common zeros, as conjectured by Greenberg. 

Part $2$ is dedicated to the estimates connected to the Birch and Swinnerton-Dyer conjectures. We bound the $p$-adic analytic rank of $A_f(\Q_\infty)$, which is the number of zeros of $L_\alpha(f^\sigma,T)$ at cyclotomic points $T=\zeta_{p^n}-1$ summed over all Galois conjugates $f^\sigma$ of $f$. This is hard even when $f$ is ordinary, since $f^\sigma$ may not be ordinary, i.e. we may have $v=0$ but $v^\sigma>0$, where $v^\sigma$ is the valuation for $a_p^\sigma$ for $\sigma\in\Gal(\overline{\Q}/\Q)$. We overcome this difficulty by giving an upper bound in terms of the Iwasawsa invariants of our pair of Iwasawa functions in the supersingular case. Apart from our Iwasawa-theoretic arguments, the key ingredient for this upper bound is to find any non-jump in the analytic rank in the cyclotomic tower, i.e. any $n$ so that $L_\alpha(f,\zeta_{p^{n}}-1)\neq0$. Note that this is a much weaker corollary to Rohrlich's theorem (stating that almost all $n$ give rise to non-jumps). This manifests itself in the fact that the first such $n$ is typically very small. We also give another upper bound that assumes $a_p=0$ (so that all $f^\sigma$ are supersingular), and is due to Pollack under the further assumption $p \equiv 3 \pmod{4}$, which our new proof removes. This upper bound is in most cases not as sharp as the more general one. Note that the upper bounds are also upper bounds for the corresponding algebraic objects (i.e. that rank of $A_f(\Q_\infty)$) in view of classical work of Perrin-Riou \cite[Lemme 6.10]{PR}%In terms of Iwasawa theory, \textit{to prove that the analytic rank of $E(\Q_\infty)$ is bounded, it suffices to prove that at least one of the Iwasawa functions be non-zero.} 
%Our methods also allow us to generalize a conjecture of Kurihara and Pollack on where the jumps in the rank occur.

We then give growth formulas for the analytic size of $\Sha(A_f/\Q_n)[p^\infty]$, unifying results of Mazur (who assumed $v^\sigma=0$ for all $\sigma$) and Pollack (who assumed $v^\sigma=\infty$ for all $\sigma$, i.e. $a_p=0$), and finishing this problem in most of the good reduction case\footnote{The remaining cases we term the \textit{sporadic cases}, which shouldn't occur in nature: For $v^\sigma>0$, we are in the sporadic case when the $\mu$-invariants differ in a specific way and $v^\sigma=\frac{p^{-k}}{2}$ for $k\in\N$ and the valuation of $\sigma(a_p)^2-\epsilon(p)\Phi_p(\zeta_{p^{k+2}})$ is exactly $2v^\sigma(1+p^{-1}-p^{-2})$ for $n$ with a fixed parity with respect to $k$, see Definition \ref{sporadic}.}. For example, we finish this problem when $A_f$ is an elliptic curve, where there are infinitely many remaining cases all for which $p=2$ or $p=3$ in view of the Hasse bound. In Mazur's case, this formula was governed by $L_\alpha(f^\sigma,T)$, while in the $a_p=0$ case, Pollack's Iwasawa functions were alternatingly responsible for the growth at even $n$ and odd $n$. The reason we can cover the remaining cases is that our estimates result from both of our Iwasawa functions working \textit{in tandem}, giving rise to several growth formula scenarios in these remaining cases, illustrating their difficulty even when $A_f$ is an elliptic curve. In the ordinary case (and some special subcases of the supersingular case), one of the Iwasawa functions dominates, and only the invariants of that function are visible in the estimates. When $0<v^\sigma<\frac{1}{2}$, we encounter a mysterious phenomenon: The estimates depend further on which one of (up to infinitely many) progressively smaller intervals $v^\sigma$ lies in, and the roles of the Iwasawa functions generally alternate in adjacent intervals. We suspect the answer to the following question is very deep: \textit{Where does this phenomenon come from and why does it occur?}

We now state our results more precisely. We work in the context of weight two modular forms and a good (coprime to the level) prime $p$. The functions $L_\alpha(f,T)$ and $L_\beta(f,T)$ are named after the roots $\alpha$ and $\beta$ of the Hecke polynomial ${X^2-a_pX+\epsilon(p)p}$, ordered so that $\ord_p(\alpha)\leq\ord_p(\beta)$. In the supersingular case (i.e. when $v:=\ord_p(a_p)>0$), we can now trace the $p$-adic $L$-functions back to a pair of of Iwasawa functions when $v=\infty$ (i.e. $a_p=0$) thanks to the methods of Pollack \cite{pollack}.

In $\textbf{Part 1}$, we prove the following theorem.
  \begin{theorem}\label{introtheorem}Let $f=\sum a_nq^n$ be a modular form of weight two and $p$ be a good prime. We let $\Lambda=\O[[T]]$, where $\O$ is the ring of integers of the completion at $p$ of $\Q((a_n)_{n\in\N},\epsilon(\Z))$).\begin{enumerate} \item When $p$ is a supersingular prime, we have
$$ \left(L_\alpha(f,T),L_\beta(f,T)\right)=\left(L_p^\sharp(f,T),L_p^\flat(f,T)\right)\Log_{\alpha,\beta}(1+T),$$
for two power series $L_p^\sharp(f,T)$ and $L_p^\flat(f,T)$ which are elements of $\Lambda$, and $\Log_{\alpha,\beta}(1+T)$ is an explicit $2\times2$ matrix of functions converging on the open unit disc.

\item When $p$ is ordinary, we can write
\[L_\alpha(f,T)=L_p^\sharp(f,T)\log_\alpha^\sharp(1+T)+L_p^\flat(f,T)\log_\alpha^\flat(1+T),\]

for some non-unique Iwasawa functions $L_p^\sharp(f,T)$ and $L_p^\flat(f,T)$, where $\log_\alpha^\sharp(T)$ and $\log_\alpha^\flat(T)$ are the entries in the first column of $\Log_{\alpha,\beta}(1+T)$. They are functions converging on the closed unit disc.
\end{enumerate}
\end{theorem}

 In the supersingular case, our vector $(L_p^\sharp(f,T),L_p^\flat(f,T))$ is related to the vector $(L_\alpha(f,T),L_\beta(f,T))$ much like the completed Riemann zeta function is related to the original zeta function: Since $L_\alpha(f,T)$ and $L_\beta(f,T)$ are not Iwasawa functions, they have infinitely many zeros in the open unit disk. The analogue of the Gamma factor is the matrix ${\Log_{\alpha,\beta}(1+T)}$. It removes zeros of linear combinations of $L_\alpha(f,T)$ and $L_\beta(f,T)$, producing the vector of Iwasawa functions with finitely many zeros. Its definition for odd $p$ is
 
 \[\Logarithm:=\lim_{n \rightarrow \infty}\CCC_1\cdots \CCC_n C^{-(n+2)}\roots\text{ , where}\]
 $\CCC_i:=\smat{ a_p & 1 \\ -\epsilon(p)\Phi_{p^i}(1+T) & 0}$, $C:=\smat{ a_p & 1 \\ -\epsilon(p)p & 0}$, and $\Phi_n(X)$ is the $n$th cyclotomic polynomial.

As one immediate corollary of Theorem \ref{introtheorem}, we obtain that $L_\alpha(f,T)$ and $L_\beta(f,T)$ have finitely many common zeros, as conjectured by Greenberg, e.g. in \cite{greenberg}.
 
%A more algebraic construction of Iwasawa functions uses Kato's zeta element. In the ordinary case, this can be done by employing Perrin-Riou's exponential map \cite{kato} to construct $L_\alpha(f,T)$. In the supersingular case, the zeta element construction of Pollack's Iwasawa function is work of Kobayashi \cite{kobayashi}. See also \cite{shuron} for the case $v\geq 1$. Using zeta elements, there is an analogue of our above Theorem \ref{introtheorem} in a more general context (addressing higher weight modular forms) due to Lei, Loeffler, and Zerbes \cite[Theorem 1.3]{llz}. Their idea is to use \textit{bases of Wach modules} (which are certain ($\phi,\Gamma)$-modules), e.g. the Berger-Li-Zhu basis \cite{blz} in the supersingular case, to construct pairs of Iwasawa functions in $\Q\otimes\Lambda$. Their methods allow to pin down a pair of Iwasawa functions in the ordinary case as well (see \cite{loefflerzerbes}). Thus, their work is a huge hint that Theorem \ref{introtheorem} should generalize to higher weight analogue as well! Their statements are not precise enough to give the applications in Part $2$, in which we heavily rely on the exact statements of Theorem \ref{introtheorem}: The first column of $\Logarithm$ appears \textit{uniformly} in the ordinary and supersingular cases to construct the pair of Iwasawa functions, and gives an \textit{explicit} matrix $\Logarithm$. This, and 
When $p=2$ and $a_2=0$, our construction of $\Logarithm$ explains a seemingly artificial extra factor of $1\over 2$ in Pollack's corresponding half-logarithm \cite{pollack}. The theorem also shows that for $p=2$, the functions $L_2^\sharp(T)$ and $L_2^\flat(T)$ of \cite{shuron} in $\Lambda\otimes\Q$ are in fact elements of $\Lambda$ in the strong Weil case.  

Our proof of Theorem \ref{introtheorem} is completely $p$-adic analytic, generalizing the arguments of Pollack \cite{pollack} when $v=\infty$ (i.e. $a_p=0$). Recall that the methods in \cite{shuron} extend to the case $v\geq \frac{1}{2}$ (\cite[Remark 5.26]{llz}). However, the situation for the \textit{remaining} (and more difficult) valuations when $v<\frac{1}{2}$ is more involved. This part forms the technical heart of the first half of the paper, in which one major new tool is Lemma \ref{zero}, which gives an explicit expansion of the terms of $\Logarithm$. We use \textit{valuation matrices}, an idea introduced in \cite{nextpaper}, to scrutinize the growth properties of the functions in its columns: They grow like $L_\alpha(f,T)$ and $L_\beta(f,T)$ for $v\geq {1\over 2}$, and at most as fast as these functions when $v<{1\over 2}$, proving e.g. that the entries in the first column are Iwasawa functions when $v=0$.

We also construct a completed version $\widehat{\Log}_{\alpha,\beta}$ of $\Log_{\alpha,\beta}$, and similarly $\widehat{L}_p^\sharp$ and $\widehat{L}_p^\flat$, and then prove functional equations for these completed objects.
\begin{theorem}Let $p$ be supersingular. Then under the change of variables $(1+T)\mapsto (1+T)^{-1}$, $\widehat{\Log}_{\alpha,\beta}(1+T)$ is invariant, and the vector $\left(\widehat{L}_p^\sharp(T),\widehat{L}_p^\flat(T)\right)$ is invariant  up to a root number of the form $-\epsilon(-1)(1+T)^{-\log_\gamma(N)}$. A similar statement holds for $p=2$.
\end{theorem}

For a precise definition of the root number, we refer to Section $2$.

We derive functional equations for $L_p^\sharp$ and $L_p^\flat$ in some cases as well, which correct a corresponding statement in \cite{pollack} (where $a_p=0$), which is off by a unit factor. The algebraic version of the functional equation by Kim \cite{kim} when $a_p=0$ is still correct, since it is given up to units.

\textbf{Part 2} is concerned with applications concerning the invariants of the Birch and Swinnerton-Dyer (BSD) conjectures in the cyclotomic direction: %They rest on the fact that $L_p^\sharp(f,T)$ and $L_p^\flat(f,T)$ are in $\Lambda$ (rather than $\Q\otimes\Lambda$). %There are two such questions, one for the behavior of $p$-adic $L$-functions at $T=0$ addressed in \cite{bsdat0}, and the other for that at the various $T=\zeta_{p^n}-1$ when $n>0$, i.e. in the cyclotomic direction. We work in the context of elliptic curves over $\Q$, unless otherwise stated. Fix such a curve $E/\Q$.

%\textit{For the behavior at the various $T=\zeta_{p^n}-1$}, there are the following aspects of rank and leading term:

$\circ$ Choose a subset $\G_f$ of $\Gal(\overline{\Q}/\Q)$ so that $\{f^\sigma\}_{\sigma\in\G_f}$ contains each Galois conjugate of $f$ once. Each zero of $L_\alpha(f^\sigma,T)$ or $L_\beta(f^\sigma,T)$ at $T=\zeta_{p^n}-1$  (counted with multiplicity) with $\sigma\in\G_f$ should, in view of BSD for number fields, contribute toward the jump in the ranks $\rank(A_f(\Q_{n}))-\rank(A_f(\Q_{n-1}))$, where $\Q_n$ is the $n$th layer in the cyclotomic $\Z_p$-extension of $\Q=\Q_0$. More specifically, the number of zeroes $r_\infty^{an}$ at all $T=\zeta_{p^n}-1$ of all $L_\alpha(f^\sigma,T)$ is an analytic upper bound for $r_\infty=\lim_{n\rightarrow \infty} \rank A_f(\Q_n)$. Denote by $r_\infty^{an}(f^\sigma)$ the $\sigma$-part of $r_\infty^{an}$, i.e. the number of such zeros of $L_\alpha(f^\sigma,T)$, so that $r_\infty^{an}=\sum_{\sigma\in\G_f}r_\infty^{an}(f^\sigma)$. When $f^\sigma$ is ordinary, $r_\infty^{an}(f^\sigma)$ is bounded by the $\lambda$-invariant $\lambda^\sigma$ of $L_\alpha(f^\sigma,T)$. 

When scrutinizing the case in which $f^\sigma$ is supersingular, $L_\alpha(f^\sigma,T)$ and $L_\beta(f^\sigma,T)$ are known to have finitely many zeroes of the form $\zeta_{p^n}-1$ by a theorem of Rohrlich. By only assuming the much weaker corollary that  $L_\alpha(f^\sigma,T)$ does not vanish at some $\zeta_{p^n}-1$, we give an explicit upper bound: 

\begin{theorem}\label{rankintro}Let $\lambda_\sharp$ and $\lambda_\flat$ be the $\lambda$-invariants of $L_p^\sharp(f^\sigma,T)$ and $L_p^\flat(f^\sigma,T)$. Put
\[\begin{array}{cllc}q_n^\sharp:=\left\lfloor\frac{p^n}{p+1}\right\rfloor & \text{ if $n$ is odd,} & \text{ and $q_n^\sharp:=q_{n+1}^\sharp$ for even $n$,}\\
q_n^\flat:=\left\lfloor\frac{p^n}{p+1}\right\rfloor & \text{ if $n$ is even,} & \text{ and $q_n^\flat:=q_{n+1}^\flat$ for odd $n$}.\end{array}\]
$$\nu_\sharp:=\text{largest odd integer }n\geq1\text{ so that } \lambda_\sharp\geq p^n-p^{n-1}-q_n^\sharp,$$
$$\nu_\flat:=\text{largest even integer }n\geq2\text{ so that } \lambda_\flat\geq p^n-p^{n-1}-q_n^\flat,$$
$$\nu:=\max(\nu_\sharp,\nu_\flat).$$

\begin{enumerate}
\item Assume $\mu_\sharp=\mu_\flat$.
Then the $\sigma$-part $r_\infty^{an}(f^\sigma)$ of the cyclotomic analytic rank $r_\infty$ for $\lim_{n\rightarrow \infty} \rank A_f(\Q_n)$ is bounded above by
\[
\min(q_\nu^\sharp+\lambda_\sharp,q_\nu^\flat+\lambda_\flat).\] 
\item For the case $\mu_\sharp\neq\mu_\flat$, there is a similar bound of the form $q_\nu^*+\lambda_*$, where $*\in\{\sharp, \flat\}$.  We refer the reader to Theorem \ref{justfornumbers} for a precise formulation.
\item When $a_p=0$, another analytic upper bound is given by $\lambda_\sharp+\lambda_\flat$.
\end{enumerate}
\end{theorem}

Let $\G_f^{ord}=\{\sigma\in\G_f:f^\sigma \text{ is ordinary} \}$ and $\G_f^{ss}=\{\sigma\in\G_f:f^\sigma \text{ is supersingular} \}.$ When $\sigma\in\G_f^{ss}$, denote by $\lambda_\natural^\sigma$ the minimum of the bounds from Theorem \ref{rankintro}.

\begin{corollary}For a prime of good reduction, $r_\infty^{an}$ is bounded above as follows: $$r_\infty^{an}\leq\sum_{\sigma\in\G_f^{ord}}\lambda^\sigma+\sum_{\sigma\in\G_f^{ss}}\lambda_\natural^\sigma$$

\end{corollary}

The \it{Kurihara terms} \rm $q_n^{\sharp/\flat}$ in Theorem \ref{rankintro} are $p$-power sums, e.g. when $\nu>1$ and $\nu$ is odd, we have 
$$q_\nu^\sharp=p^{\nu-1}-p^{\nu-2}+p^{\nu-3}-p^{\nu-4}+\cdots+p^2-p.$$ The $\nu\in\N$ is chosen according to an explicit algorithm that measures the contribution of $\Logarithm$ to the cyclotomic zeroes. %, and the bound becomes $\min(\lambda_\sharp+q_\nu^\sharp,\lambda_\flat+q_\nu^\flat)$. 
For example, when $\mu_\sharp=\mu_\flat$ and $\lambda_\sharp<p-1$ and $\lambda_\flat<(p-1)^2$, we have $\nu=0$, in which case $q_0^\sharp=q_0^\flat=0$ and the bound is simply $\min(\lambda_\sharp,\lambda_\flat)$, which is very much in the spirit of the bound in the ordinary case.

The bound for the case $a_p=0$, $\lambda_\sharp+\lambda_\flat$, is a generalization of work of Pollack \cite{pollack}. When $p$ is odd, this bound is in most (computationally known) cases weaker than the above one, but there are cases in which it is stronger. It is interesting to ask for an \textit{optimal bound}.
%, and we believe that to answer this question, one would have to at least modify our methods quite significantly.

$\circ$ For the leading term part, we know from above that $L_\alpha(f,T)$ and $L_\beta(f,T)$ don't vanish at $T=\zeta_{p^n}-1$ for $n \gg 0,$ so that these values should encode $\#(\Sha(A_f/\Q_n)[p^\infty])/\#(\Sha(A_f/\Q_{n-1})[p^\infty])$, i.e.  the jumps in the $p$-primary parts of $\Sha$ at the $n$th layer of the cyclotomic $\Z_p$-extension. 

In the ordinary case, a classical result of Mazur gives an estimate for 
$$\#\Sha^{an}(A_f/{\Q_n}):=\frac{L^{(r_n^{an'})}(A_f/{\Q_n},1)\#A_f^{tor}(\Q_n)\#\widehat{A}_f^{tor}(\Q_n)\sqrt{D(\Q_n)}}{(r_n^{an'})!\Omega_{A_f/{\Q_n}}R(A_f/{\Q_n})\Tam(A_f/{\Q_n})}.$$

His analytic estimate for $e_n:=\ord_p(\#\Sha^{an}(A_f/{\Q_n}))$ when $f^\sigma$ are all ordinary says that for $n \gg 0$,
$$e_n-e_{n-1}=\sum_{\sigma\in\G_f} \mu^\sigma(p^n-p^{n-1})+\lambda^\sigma-r_\infty^{an}(f^\sigma),$$

much in the spirit of Iwasawa's famous class number formula. Here, $\mu^\sigma$ and $\lambda^\sigma$ are the Iwasawa invariants of $L_\alpha(f^\sigma,T)$. We prove a theorem that estimates $e_n$ in the general good reduction case in terms of the Iwasawa invariants of $L_p^\sharp(f^\sigma,T)$ and $L_p^\flat(f^\sigma,T)$ and $v^\sigma=\ord_p(a_p^\sigma$): 

Given an integer $n$, we now define two generalized Kurihara terms $q_n^*(v^\sigma)$ for $*\in\{\sharp,\flat\}$ which are \textit{continuous} in $v^\sigma\in[0,\infty]$. They are each a sum of a truncated Kurihara term and a multiple of $p^n-p^{n-1}$. For fixed $n$, they are piecewise linear in $v$.
\begin{definition} For a real number $v>0$, let $k\in \Z^{\geq 1}$ be the smallest positive integer so that $v\geq{p^{-k}\over 2}$.  

$$q_n^\sharp(v):=\begin{cases}(p^n-p^{n-1})kv+\left\lfloor\frac{p^{n-k}}{p+1}\right\rfloor &\text{ when $n\not\equiv k\mod(2)$}\\(p^n-p^{n-1})\left((k-1)v\right)+\left\lfloor\frac{p^{n+1-k}}{p+1}\right\rfloor&\text{ when $n\equiv k\mod(2)$},\end{cases}$$
$$q_n^\flat(v):=\begin{cases}(p^n-p^{n-1})\left((k-1)v\right)+p\left\lfloor\frac{p^{n-k}}{p+1}\right\rfloor+p-1&\text{ when $n\not\equiv k\mod(2)$}\\(p^n-p^{n-1})kv +p\left\lfloor\frac{p^{n-1-k}}{p+1}\right\rfloor+p-1&\text{ when $n\equiv k\mod(2)$}.\end{cases}$$

We also put $$ q_n^*(\infty):=\lim_{v\rightarrow \infty}q_n^*(v) \text{ and } q_n^*(0):=\lim_{v\rightarrow 0}q_n^*(v)=\begin{cases}0&\text{ when $*=\sharp$}\\p-1&\text{ when $*=\flat$.}\end{cases}$$
\end{definition}

%(We use this seemingly strange adherence to the symbol $v_2$ simply for uniform notation.)

%\end{definition}
%\definition The \textbf{sporadic case} occurs if $v=0$ and $\mu_\sharp=\mu_\flat$ and $\lambda_\sharp=\lambda_\flat+p-1$,
%{or if $v=\frac{p^{-k}}{2} $ and $ v_2=2v(1+p^{-1}-p^{-2})$ and}$$\begin{cases}n\not\equiv k \mod{2} \text{ and }\begin{cases}\text{$\mu_\sharp-\mu_\flat>v-\frac{2v}{p^3+p^2}$}\text{ or }\\\text{$\mu_\sharp-\mu_\flat=v-\frac{2v}{p^3+p^2}$ and $\lambda_\sharp>\lambda_\flat$, or}\end{cases} \\ n\equiv k \mod{2} \text{ and } \begin{cases}\text{$\mu_\sharp-\mu_\flat<\frac{2v}{p^3+p^2}-v$}\text{ or }\\\text{$\mu_\sharp-\mu_\flat=\frac{2v}{p^3+p^2}-v$ and $\lambda_\sharp\leq\lambda_\flat$.}\end{cases} \end{cases}$$
%\end{definition}

\definition[(Modesty Algorithm)]\label{modestyintro}
Given $v\in[0,\infty]$, an integer $n$, integers $\lambda_\sharp$ and $\lambda_\flat$, and rational numbers $\mu_\sharp$ and $\mu_\flat$, choose $*\in\{\sharp,\flat\}$ via
$$*=\begin{cases}\sharp \text{ if $(p^n-p^{n-1})\mu_\sharp+\lambda_\sharp+q_n^\sharp(v)<(p^n-p^{n-1})\mu_\flat+\lambda_\flat+q_n^\flat(v)$}\\\flat \text{ if $(p^n-p^{n-1})\mu_\flat+\lambda_\flat+q_n^\flat(v)<(p^n-p^{n-1})\mu_\sharp+\lambda_\sharp+q_n^\sharp(v)$.}\end{cases}$$
 %\end{definition}

\begin{theorem}
%\footnote{In Theorem \ref{specialvalues}, we estimate a related term not containing $r_\infty$, the number of zeroes of $L_\alpha(f,T)$ of the form $T=\zeta_{p^n}-1$.}]
Let $e_n:=\ord_p(\#\Sha^{an}(A_f/{\Q_n}))$. Then for $n\gg0$, we have

$$e_n-e_{n-1}=\sum_{\sigma\in\G_f}\left( \mu_*^\sigma(p^n-p^{n-1})+\lambda_*^\sigma+q_n^{*}(v^\sigma)\right)-r_\infty^{an}, \text{ where }$$

 \begin{enumerate}
\item
$\ast\in\{\sharp,\flat\}$ is chosen according to the Modesty Algorithm \ref{modestyintro} with the choice $v=v^\sigma$ when ${v^\sigma \neq \frac{p^{-k}}{2}}$. Note that one input of the algorithm is the parity of the integer $k$ so that $v^\sigma\in[\frac{1}{2}p^{-k},\frac{1}{2}p^{-k+1})$. 
\item the term $q_n^*(v^\sigma)$ is replaced by a modified Kurihara term $q_n^*(v^\sigma,v_2^\sigma)$ when $v^\sigma = \frac{p^{-k}}{2}$ that depends further on the valuation $v_2^\sigma$ of $(a_p^{\sigma})^2-\epsilon(p)\Phi_p(\zeta_{p^{k+2}})$ when $\mu_\sharp^\sigma\neq\mu_\flat^\sigma$. Further, $\ast\in\{\sharp,\flat\}$ is chosen according to a generalized Modesty Algorithm which also depends on $v_2^\sigma$. We refer the reader to Theorem \ref{specialvalues} for a precise formulation.
\end{enumerate}
\end{theorem}
\begin{remark}The (generalized) Modesty Algorithm doesn't work for some excluded cases (``sporadic''). When $v^\sigma=0$, the excluded case is  $\mu_\sharp^\sigma=\mu_\flat^\sigma$, and $\lambda_\sharp^\sigma=\lambda_\flat^\sigma+p-1$, which can be remedied by adhering to the ordinary theory, see Theorem \ref{specialvalues}. The other excluded cases occur when $v^\sigma= \frac{p^{-k}}{2}$ and $v_2^\sigma=p^{-k}(1+p^{-1}-p^{-2})$ and an inequality of $\mu$-invariants (or $\lambda$-invariants) is satisfied. This case should conjecturally not occur. See Definition \ref{sporadic} for details.
\end{remark}
\rm

%The terms $q_n^\sharp(v,v_2)$ and $q_n^\flat(v,v_2)$ are each a sum of a truncated Kurihara term and a multiple of $p^n-p^{n-1}$. For fixed $n$, they are piecewise linear in $v$, away from break points at $v=\frac{1}{2}p^{-k}$ where they depend on another explicit valuation $v_2$ if $\mu_\sharp\neq \mu_\flat$. The terms approach $0$ as $v$ does.%Note that the tail terms $\left\lfloor\frac{p^{n-k}}{p+1}\right\rfloor$ (resp. $\left\lfloor\frac{p^{n+1-k}}{p+1}\right\rfloor$) appearing in the generalized Kurihara terms $q_n^\sharp(v)$ are equal to the classical Kurihara terms $q_{n-k}^\sharp$. For $n>k$, those for $q_n^\flat(v)$, i.e. $p\left\lfloor\frac{p^{n-k}}{p+1}\right\rfloor+p-1$ and $p\left\lfloor\frac{p^{n-1-k}}{p+1}\right\rfloor+p-1$, are both $q_{n-k}^\flat$.

In less precise terms, the theorem above states that \textit{our formulas in the supersingular case approach Mazur's formula in the ordinary case as $k\rightarrow \infty$}, and that \textit{during this approach, the roles of $\sharp$ and $\flat$ may switch as the parity of $k$ does}. The simplest scenario is when the $\mu$-invariants are equal. Here, a switch in the parity of $k$ \textit{always} causes a switch in the role of $\sharp$ and $\flat$. 
It is a mystery why these formulas appear in this way, but we invite the reader to ponder this phenomenon by looking at Figure $1$.
\begin{figure}[ht]
\centering
\begin{tabular}{rrr}
\multicolumn{2}{l}{} \\ &
   \multirow{21}{*}{\includegraphics[height=310pt]{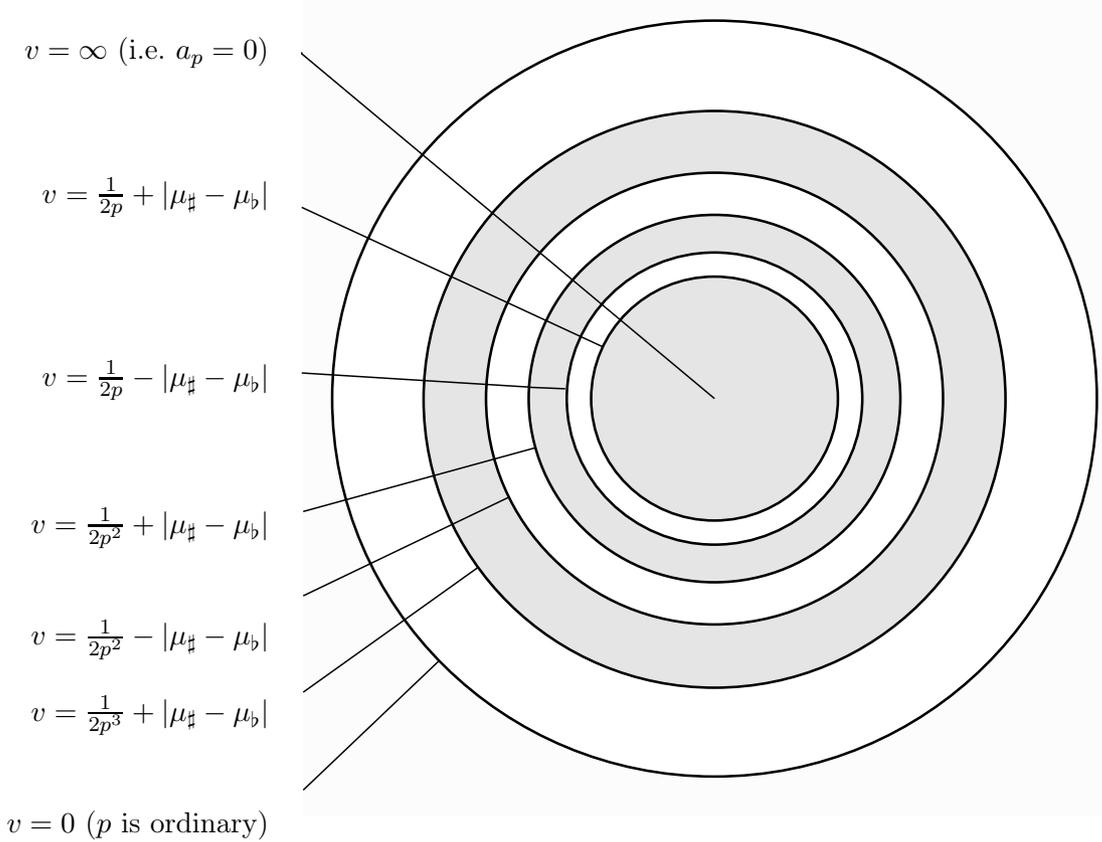}} \\ $v=\infty$ (i.e. $a_p=0$) &  \\   \\ &  \\
 &  \\$v=\frac{1}{2p}+\left|\mu_\sharp-\mu_\flat\right|$ &  \\ \\ &  \\ &  \\ &   \\ $v=\frac{1}{2p}-|\mu_\sharp-\mu_\flat|$&  \\ &  \\  &  \\ &  \\ $v=\frac{1}{2p^2}+|\mu_\sharp-\mu_\flat|$  &  \\ \\ \\$v=\frac{1}{2p^2}-|\mu_\sharp-\mu_\flat|$ \\  \\$v=\frac{1}{2p^3}+|\mu_\sharp-\mu_\flat|$ \\ \\ \\   $v=0$ ($p$ is ordinary) \\ \\
\end{tabular}
\caption{\footnotesize{The locus inside the $p$-adic unit disc in which the modesty algorithm chooses $\sharp$ or $\flat$ when $\frac{p-1}{4p^4}<|\mu_\sharp-\mu_\flat|<\frac{p-1}{4p^3}$.
Here, $v=\ord_p(a_p)$ indicates the possible valuations of $a_p$ inside the unit disc. At the center, we have $a_p=0$, i.e. $v=\infty$. On the edge, we have $v=0$, so that $p$ is ordinary.
In the central shaded region, the formula involves $\mu_\sharp$, $\lambda_\sharp$ for odd $n$ and $\mu_\flat$, $\lambda_\flat$ for even $n$, while in the second shaded region, $\mu_\sharp$ and $\lambda_\sharp$ are part of the formula for even $n$, and $\mu_\flat$ and $\lambda_\flat$ for odd $n$. In the outermost shaded region, the roles are flipped yet again and the $\sharp$-invariants come into play for odd $n$, and the $\flat$-invariants for even $n$. 
When $\mu_\sharp<\mu_\flat$, the formulas are only controlled by the $\mu_\sharp$ and $\lambda_\sharp$ in the non-shaded regions.}}
\end{figure}

In the supersingular case, Greenberg, Iovita, and Pollack (in unpublished work around $2005$) generalized the approach of Perrin-Riou of extracting invariants $\mu_\pm$ and $\lambda_\pm$ from the classical $p$-adic $L$-functions for a modular form $f$, $L_p(f,\alpha,T)$ and $L_p(f,\beta,T)$, which they used for their estimates (under the assumption $\mu_+=\mu_-$). Our formulas match theirs exactly in those cases, although the techniques are different.

We write out explicitly the elliptic curve case of the above theorem for the convenience of the readers, and since it hints at a unification of the ordinary and supersingular theories: 

\begin{corollary}\label{shagrowth} Let $E$ be an elliptic curve over $\Q$, $p$ a prime of good reduction, $v=\ord_p(a_p)$ and $e_n:=\ord_p(\#\Sha^{an}(E/{\Q_n}))$, and $\mu_{\sharp/\flat}$ and $\lambda_{\sharp/\flat}$ the Iwasawa invariants of $L_p^{\sharp/\flat}(E,T)$. Then for $n\gg0$,
$$e_n-e_{n-1}=\mu_*(p^n-p^{n-1})+\lambda_*-r_\infty^{an}+\min(1,v)q_n^*,$$
where $q_n^*$ are the Kurihara terms from Theorem \ref{rankintro} and $\ast\in \{\sharp,\flat\}$ is chosen as follows:

\begin{center}
    \begin{tabular}{ | c || c | c | c | c | c | c | c|}
    \hline
    \multirow{2}{*}{}& \multicolumn{3}{|c|}{$v=0$} &  \multicolumn{2}{|c|}{$0<v<\infty$}& \multicolumn{2}{|c|}{$v=\infty$}\\
     \hhline{|~H-------}
     &  $\lambda_\sharp<\lambda_\flat'$ & $\lambda_\sharp>\lambda_\flat'$&$\lambda_\sharp=\lambda_\flat'$& $n$ odd & $n$ even &$n$ odd &  $n$ even \\
    \hhline{|=#=|=|=|=|=|=|=|}
    $\mu_\sharp=\mu_\flat$&$*=\sharp$&$*=\flat$&\tiny{excluded}&$*=\sharp$&$*=\flat$& \multirow{3}{*}{$*=\sharp$}&\multirow{3}{*}{$*=\flat$}\\ \hhline{------~~}
    $\mu_\sharp<\mu_\flat$ & \multicolumn{5}{|c|}{$*=\sharp$}&&\\ \hhline{------~~}
    $\mu_\flat<\mu_\sharp$ &\multicolumn{5}{|c|}{$*=\flat$}&&\\ \hline
    \end{tabular}
\end{center}
Here, we have denoted $\lambda_\flat+p-1$ by $\lambda_\flat'$. 
\end{corollary}

In particular, there are \textit{three} different possible formulas for the growth of the \v{S}afarevi\v{c}-Tate group when $a_p\neq0$ and $p$ is supersingular, one for each scenario of comparing $\mu$-invariants. As visible when $a_p\neq0$, the \v{S}afarevi\v{c}-Tate group $\Sha$ tries to stay as small as possible (``it is modest'') during its ascent along the cyclotomic $\Z_p$-extension by choosing smaller Iwasawa-invariants. The analytic estimates in the case $v=\infty$ (i.e. $a_p=0)$ were given by Pollack in \cite{pollack}.

Thanks to the work of Kurihara \cite{kurihara}, Perrin-Riou \cite{perrinriou}, Kobayashi \cite{kobayashi}, and the work in \cite{nextpaper}, we now understand the algebraic side of the corollary (i.e. the elliptic curve case) quite well in the supersingular case when $p$ is odd \footnote{These works answer a comment by Coates and Sujatha who wrote in their textbook \cite[page 56]{coatessujatha} only 15 years ago that when looking at the $p$-primary part of $\Sha(E/\Q_n)$ as $n\rightarrow \infty$,

\begin{center}{\textit{``...nothing seems to be known about the asymptotic behavior of the order..."} }\end{center}}. The formulas also match the algebraic ones of Kurihara and Otsuki when $p=2$ \cite{kuriharaotsuki}. For the unknown cases (in which $p=2$), the formulas thus serve as a prediction of how $\Sha(E/\Q_n)[p^\infty]$ grows.

\subsection{Organization of Paper}\label{organizationofpaper} This paper consists of two parts. Part $1$ is mainly concerned with the construction of our pair of Iwasawa functions: In Section $2$, we introduce \textit{Mazur-Tate symbols}, which inherit special values of $L$-functions to construct the classical $p$-adic $L$-functions of Amice, V\'{e}lu, and Vi\v{s}ik, and state the main theorem. In Section $3$, we give a quick application, answering a question by Greenberg. In Section $4$, we scrutinize the logarithm matrix $\Log_{\alpha,\beta}$ and prove its basic properties. In Section $5$, we then put this information together to rewrite the $p$-adic $L$-functions from Section $2$ in terms of the new $p$-adic $L$-functions $L_p^\sharp$ and $L_p^\flat$, proving the main theorem. Part $2$ is devoted to the BSD-theoretic aspects as one climbs up the cyclotomic tower: In Section $6$, we give the necessary preparation, Section $7$ is concerned with the two upper bounds on the Mordell-Weil rank, and Section $8$ scrutinizes the size of $\Sha$. %We give the proofs for the elliptic curve case separately before going into the more complex modular form case.
\subsection{Outlook}\label{outlook}  % to construct pairs of $p$-adic $L$-functions whose power series coefficients have bounded growth \cite{llz}. Their analogue of $\Log_{\alpha,\beta}(1+T)$ is not constructed directly, so that one may ask for an explicit description, ideally in terms of $\Log_{\alpha,\beta}(1+T)$. An answer will hopefully show that their $p$-adic $L$-functions are in $\Lambda$ like ours. Their methods use the basis of Berger, Li, and Zhu \cite{blz} in the supersingular case, but Loeffler and Zerbes use a {\it separate } basis for the Wach module in the ordinary case in \cite{loefflerzerbes}, thus arriving at a {\it different} analogue of $\Log_{\alpha,\beta}(1+T)$ and a pair of Iwasawa functions, one of which is a scaled version of $L_\alpha(f,T)$. A {\it uniform} construction of Wach modules applicable to both the ordinary and supersingular scenarios should give rise to a uniform analogue of $\Log_{\alpha,\beta}(1+T)$. The main theorem in this paper suggests that something in this direction is possible, since the first column of $\Logarithm$ appears uniformly in it. The hope is then to pin down a natural choice of $L_p^\sharp(f,T),L_p^\flat(f,T)$ in the ordinary case as well.

In \cite{jay}, Pottharst constructs an algebraic counterpart (\textit{Selmer modules}) to the pair $L_\alpha(f,T),L_\beta(f,T)$ in the supersingular case, which hints at an algebraic counterpart of $\Logarithm$ as well, along with algebraic versions of each of our analytic applications, equivalent under an Iwasawa Main Conjecture. A proof of the Main Conjecture in terms of $L_\alpha(f,T)$ in the ordinary case is due to \cite{skinnerurban}, building on \cite{kato}. See also \cite{rubin} for the CM case. The work of Lei, Loeffler, and Zerbes \cite{llz} constructs pairs of Iwasawa functions (in $\Q\otimes\Lambda$) out of Berger-Li-Zhu's basis \cite{blz} of \textit{Wach modules}, cf. also \cite{loefflerzerbes}. They match the Iwasawa functions in this paper when $a_p=0$ (which shows that the functions in \cite{llz} actually live in $\Lambda$), which hints at an explicit relationship between our methods and theirs.
{For the higher weight case, there are generalizations of $L_p^\sharp(f,T)$ and $L_p^\flat(f,T)$, which is forthcoming work. Their invariants are already sometimes visible (see e.g. \cite{pollackweston}). It would be nice to generalize the pairs of $p$-adic BSD conjectures to modular abelian varieties as well. For the ordinary case, the generalization of \cite{mtt} is \cite{BMS}. Another challenge is formulating $p$-adic BSD for a bad prime. Apart from \cite{mtt}, a hint for what to do can be found in Delbourgo's formulation of Iwasawa theory \cite{delbourgo}. }

\part{The pair of Iwasawa functions $L_p^\sharp$ and $L_p^\flat$}

\section{The $p$-adic $L$-function of a modular form}\label{padiclfunction}
In this section, we recall the classical $p$-adic $L$-functions given in \cite{amicevelu}, \cite{mtt}, \cite{vishik}, and \cite{msd}, in the case of weight two modular forms. We give a construction via \textbf{queue sequences} which we scrutinize carefully enough to arrive at the decomposition of the classical $p$-adic $L$-functions as linear combinations of Iwasawa functions $L_p^\sharp$ and $L_p^\flat$. This is the main theorem of this paper, which will be proved in Sections $4$ and $5$, and upon which the applications (Sections $3,6,$ and $7$) depend.

Let $f$ be a weight two modular form with character $\epsilon$ which is an eigenform for the Hecke operators $T_n$ with eigenvalue $a_n, K(f)$ the number field $\Q((a_n)_{n\in\N},\epsilon(\Z))$ and $\O(f)$ its ring of integers. We also fix forever a good (i.e. coprime to the level) prime $p$. Given integers $a,m$, the \textit{period} of $f$ is
$$\phi\left(f,{a\over m}\right):=2\pi i\int_{i\infty}^{a\over m} f(z)dz.$$
The following theorem puts these transcendental periods into the algebraic realm.
\begin{theorem}\label{integrality}There are nonzero complex numbers $\Omega_f^\pm$ so that the following are in $\O(f)$ for all $a,m\in\Z$:
$$\left[{a\over m}\right]_f^+:=\frac{\phi(f,\frac{a}{m})+\phi(f,{-a\over m})}{2\Omega_f^+} \text{ and } \left[{a\over m}\right]_f^-:=\frac{\phi(f,\frac{a}{m})-\phi(f,{-a\over m})}{2\Omega_f^-}. $$
 \end{theorem}
\proof \cite[page 375]{maninparabolic}, or \cite[Theorem 3.5.4]{greenbergstevens}.

\begin{convention}[(For Part II)] We make the choice convention that $\prod_{f^\sigma}\Omega_{f^\sigma}^\pm=\Omega_{A_f}^\pm$, where the $f^\sigma$ run over all Galois conjugates of $f$ (see e.g. \cite[(2.4)]{BMS}), and $\Omega_{A_f}^\pm$ are the Neron periods as in \cite[\S 8.10]{manin}. We also make the convention that any period $\Omega$ with an omitted sign denotes $\Omega^+$.
\end{convention}
$\left[{a\over m}\right]_f^\pm$ are called \textit{modular symbols}. For $p$-adic considerations, we fix an embedding $\overline{\Q}\rightarrow \C_p$ of an algebraic closure of $\Q$ inside the completion of an algebraic closure of $\Q_p$. This all allows us to construct $p$-adic $L$-functions as follows:
Denote by $\C^0(\Z_p^\times)$ the $\C_p$-valued step functions on $\Z_p^\times$. Let $a$ be an integer prime to $p$, and denote by $\mathbf{1}_U$ the characteristic function of an open set $U$. We let $\ord_p$ be the valuation associated to $p$ so that $\ord_p(p)=1$. Let $\alpha$ be a root of the Hecke polynomial $X^2-a_pX+\epsilon(p)p$ of $f$ so that $\ord_p(\alpha)<1$, and denote the conjugate root by $\beta$. We define a linear map $\mu_{f,\alpha}^\pm$ from $\C^0(\Z_p^\times)$ to $\C_p$ by setting
$$\mu_{f,\alpha}^\pm(\mathbf{1}_{a+p^n\Z_p})=\frac{1}{\alpha^{n+1}}\left(\left[{a \over p^n}\right]_f^\pm, \left[{a \over p^{n-1} }\right]_f^\pm\right)\hidari \alpha \\ -\epsilon(p)\migi.$$
\begin{remark} The maps $\mu_{f,\alpha}^\pm$ are not measures, but $\ord_p(\alpha)$-admissible measures. See e.g. \cite{pollack}. For background on measures, see \cite[Section 12.2]{washington}.
\end{remark}
\begin{theorem}We can extend the maps $\mu_{f,\alpha}^\pm$ to all analytic functions on $\Z_p^\times$.
\end{theorem} This can be done by locally approximating analytic functions by \textit{step functions}, since $\mu_{f,\alpha}^\pm$ are $\ord_p(\alpha)$-admissible measures. That is, we look at their Taylor series expansions and ignore the higher order terms. For an explicit construction, see \cite{amicevelu} or \cite{vishik}. Since characters $\chi$ of $\Z_p^\times$ are locally analytic functions, we thus obtain an element $$L_p(f,\alpha,\chi):= \mu_{f,\alpha}^{\sign(\chi)}(\chi).$$

Now since $\Z_p^\times\cong (\Z/2p\Z)^\times\times(1+2p\Z_p)$, we can write a character $\chi$ on $\Z_p^\times$ as a product
$$\chi=\omega^i\chi_u$$
with $0\leq i<|\Delta|$ for some $u\in\C_p$ with $|u-1|_p<1$, where $\chi_u$ sends the topological generator $\gamma=1+2p$ of $1+2p\Z_p$ to $u$, and where $\omega: \Delta \rightarrow \Z_p^\times \in \C_p$ is the usual embedding of the $|\Delta|$-th roots of unity in $\Z_p$ so that $\omega^i$ is a tame character of $\Delta=(\Z/2p\Z)^\times$. Using this product, we can identify the open unit disc of $\C_p$ with characters $\chi$ on $\Z_p^\times$ having the same tame character $\omega^i$. Thus if we fix $i$, we can regard $L_p(f,\alpha,\omega^i\chi_u)$ as a function on the open unit disc. We can go even further:

\begin{theorem}[(\cite{vishik}, \cite{mtt}, \cite{amicevelu}, \cite{pollack})] Fix a tame character $\omega^i:\Delta=(\Z/2p\Z)^\times\rightarrow \C_p$. Then the function
$L_p(f,\alpha,\omega^i\chi_u)$ is an analytic function converging on the open unit disc.
\end{theorem}
We can thus form its power series expansion about $u=1$. For convenience, we change variables by setting $T=u-1$ and denote $L_p(f,\alpha,\omega^i\chi_u)$ by $L_p(f,\alpha,\omega^i,T)$. %If $i=0$ so that $\omega^i=\mathbf{1}$ is the trivial character, we denote this function simply by $L_p(f,\alpha,T)$.

Denote by $\zeta=\zeta_{p^n}$ a primitive $p^n$th root of unity. We can then regard $\omega^i\chi_\zeta$ as a character of $(\Z/p^N\Z)^\times,$ where $N=n+1$ if $p$ is odd and $N=n+2$ if $p=2$. Given any character $\psi$ of $(\Z/p^N\Z)^\times,$ let $\tau(\psi)$ be the Gau\ss{ }sum $\sum_{a\in(\Z/p^N\Z)^\times}\psi(a)\zeta_{p^N}^a$.

%\definition\label{interpolateswell} Let $L_p(f,\xi,\omega^i,T)$ be a power series converging on the open unit disc associated to $f$ and $\omega^i$ as above, and to a root $\xi$ of the Hecke polynomial. We say that $L_p(f,\xi,\omega^i,T)$ \textit{interpolates well} if %\end{definition}

\begin{theorem}[(Amice-V\'{e}lu  \cite{amicevelu}, Vi\v{s}ik \cite{vishik})]\label{vishik} The above $L_p(f,\alpha,\omega^i,T)$ interpolate as follows:

\resizebox{.95\hsize}{!}{$L_p(f,\alpha,\omega^i,\zeta-1)={p^N \over \alpha^N \tau\left(\omega^{-i}\chi_{\zeta^{-1}}\right)}{L\left(f_{\omega^{-i}\chi_{\zeta^{-1}}},1\right) \over \Omega_f^{\omega^i(-1)}} \text {  if $i\neq 0 \neq \zeta-1$, and } L_p(f,\alpha,\omega^0,0)=\left(1-\frac{1}{\alpha}\right)^2\frac{L(f,1)}{\Omega_f^+}
  .$}

%$$L_p(f,\alpha,0)=\left(1-\frac{1}{\alpha}\right)^2\frac{L(f,1)}{\Omega_f^+}
% \text{ and }  L_p(f,\alpha,\omega^i,\zeta-1)={1\over \alpha^N}{p^N \over \tau\left(\omega^{-i}\chi_{\zeta^{-1}}\right)}{L\left(f_{\omega^{-i}\chi_{\zeta^{-1}}},1\right) \over \Omega_f^{\omega^i(-1)}}.$$
\end{theorem}

%CHECK IN POLLACKSTHESIS $???L_p(f,\alpha,\chi_u)=\int_{\Z_p^\times}\chi_{T+1}d\mu_{f,\alpha}^{sign(\chi)}$.

\subsection{Queue sequences and Mazur-Tate elements}
Denote by $\mu_{p^n}$ the group of $p^n$th roots of unity, and put $\G_N:=\Gal(\Q(\mu_{p^N}))$. We let $\Q_n$ be the unique subextension of $\Q(\mu_{p^N})$ with Galois group isomorphic to $\Z/p^n\Z$ and put $\Gamma_n:=\Gal(\Q_n/\Q).$ We also let $\Gamma:=\Gal(\bigcup_n \Q_n/\Q)$. %Via character decompositions (NO. BUT SEND THE TOP GENERATOR TO STUFF), w
We then have an isomorphism $$\G_N\cong\Delta\times\Gamma_n.$$
We let $K:=K(f)_v$ be the completion of $K(f)$ by the prime $v$ of $K(f)$ over $p$ determined by $\ord_p(\cdot)$ and denote by $\O$ the ring of integers of $K$.
Let $\Lambda_n= \O[\Gamma_n]$ be the finite version of the Iwasawa algebra at level $n$.
We need two maps $\nu=\nu_{n-1/n}$ and $\pi=\pi_{n/n-1}$ to construct {\bf{queue sequences}}: $\pi$ is the natural projection from $\Lambda_n$ to $\Lambda_{n-1}$, and the map
$\Lambda_{n-1} \xrightarrow{\nu_{n-1/n}}\Lambda_n$ we define by $\nu_{n-1/n}(\sigma)=\displaystyle \sum_{\tau \mapsto \sigma, \tau \in \Gamma_n} \tau$.
 We let $\Lambda = \O[[\Gamma]]=\varprojlim_{\pi_{n/n-1}} \O[\Gamma_n]$ be the Iwasawa algebra. We identify $\Lambda$ with $\O[[T]]$ by sending our topological generator $\gamma=1+2p$ of $\Gamma\cong \Z_p$ to $1+T$. This induces an isomorphism between $\Lambda_n$ and $\O[[T]]/((1+T)^{p^n}-1)$.

\begin{definition}\label{queuesequence}A {\bf{queue sequence}} is a sequence of elements $(\Theta_n)_n \in (\Lambda_n)_n$ so that
$$\pi\Theta_n=a_p\Theta_{n-1}-\epsilon(p)\nu\Theta_{n-2} \text{ when $n\geq2$}.$$
\end{definition}
%WE CAN ALSO THINK OF THEM AS POLYNOMIALS.
%Convergence, done after proposition.

%\begin{proposition}Let $(\Theta_n)_n$ be a queue sequence, and put
%$$\psi_{n,\alpha}:=(\Theta_n,\nu\Theta_{n-1})\hidari \alpha^{-n}\\ -\epsilon(p)\alpha^{-(n+1)}\migi.$$
%Then $\psi_\alpha:=\lim_{n\rightarrow \infty}\psi_{n,\alpha}$ exists and converges on the open unit disc.
%\end{proposition}
%\begin{proof}Come up with own or use Pollack-Weston
%\end{proof}

\definition\label{log}For $a\in\G_N$, denote its projection onto $\Delta$ by $\overline{a}$, and let $i:\Delta\hookrightarrow \G_N$ be the standard inclusion, so that $\frac{a}{i(\overline{a})}\in\Gamma_n.$ Define $\log_\gamma(a)$ to be the smallest positive integer so that the image of $\gamma^{\log_\gamma(a)}$ under the projection from $\Gamma$ to $\Gamma_n$ equals $\frac{a}{i(\overline{a})}$. We then have a natural map $i:\Delta\hookrightarrow\G_\infty$ which allows us to extend this definition to any ${a\in \G_\infty=\varprojlim_N\G_N}$: Let $\log_\gamma(a)$ be the unique element of $\Z_p^\times$ so that $\gamma^{\log_\gamma(a)}=\frac{a}{i(\overline{a})}$.
%\end{definition}

\begin{example}We make the identification $\G_N\cong (\Z/p^N\Z)^\times$ by identifying $\sigma_a$ with $a$, where $\sigma_a(\zeta)=\zeta^a$ for $\zeta \in \mu_{p^N}$. This allows us to construct the \textbf{Mazur-Tate element},  which is the following element:
$$\vartheta_N^\pm:=\sum_{a\in (\Z/{p^N})^\times}\left[\frac{a}{p^N}\right]_f^\pm\sigma_a\in \O[\G_N].$$
For each character $\omega^i:\Delta\rightarrow \C_p^\times$, put
$$\varepsilon_{\omega^i}=\frac{1}{\#\Delta} \sum_{\tau\in \Delta} \omega^i(\tau)\tau^{-1}.$$
We can take isotypical components $\varepsilon_{\omega^i}\vartheta_N$ of the Mazur-Tate elements, which can be regarded as elements of $\Lambda_n\cong\O[[T]]/((1+T)^{p^n}-1)$. Denote these \textbf{Mazur-Tate elements associated to the tame character $\omega^i$} by $$\theta_n(\omega^i,T):=\varepsilon_{\omega^i}\vartheta_N^{\sign(\omega^i)}.$$ We extend $\omega^i$ to all of $(\Z/p^N\Z)^\times$ by precomposing with the natural projection onto $\Delta$, and can thus write these elements explicitly as elements of $\Lambda_n$:
$$\theta_n(\omega^i,T)=\sum_{a\in(\Z/p^N\Z)^\times} \left[\frac{a}{p^N}\right]_f^{\sign(\omega^i)}\omega^i(a)(1+T)^{\log_\gamma(a)}$$
When $\omega^i=\mathbf{1}$ is the trivial character, we simply write $\theta_n(T)$ instead of $\theta_n(\mathbf{1},T)$. For a fixed tame character $\omega^i$, the associated Mazur-Tate elements $\theta_n(\omega^i,T)$ form a queue sequence. For a proof, see \cite[(4.2)]{mtt}.
\end{example}

We can now explicitly approximate $L_p(f,\alpha,\omega^i,T)$ by Riemann sums:
\definition\label{riemannsumapproximation}Put
$$L_{N,\alpha}^\pm:=\sum_{a\in(\Z/p^N\Z)^\times}\mu_{f,\alpha}^\pm(\mathbf{1}_{a+p^N\Z_p})\sigma_a\in\C_p[\G_N],$$ so we get the representation
% Note that then we have $$\alpha^{N+1}L_{N,\alpha}^\pm\in\O[[\G_N]]$$ and
$$ \varepsilon_{\omega^i}L_{N,\alpha}^{\sign(\omega^i)}(T)=\sum_{a\in(\Z/p^N\Z)^\times} \mu_{f,\alpha}^{\sign(\omega^i)}(\mathbf{1}_{a+p^n\Z_p})\omega^i(a)(1+T)^{\log_\gamma(a)}. $$
%\end{definition}
Note that the homomorphism $\nu:\Gamma_{n-1}\rightarrow\Gamma_{n}$ extends naturally to a homomorphism from $\G_{N-1}$ to $\G_{N}$, also denoted by $\nu$.
%In the following lemma, we need the N \geq 1 condition because of the prime 2
\begin{lemma}\label{MTelementscomein}Let $n\geq 0$, i.e. $N\geq1$ for odd $p$, and $N\geq 2$ for $p=2$. Then$$(L_{N,\alpha}^\pm,L_{N,\beta}^\pm)=(\vartheta_N^\pm,\nu\vartheta_{N-1}^\pm)\smat{\alpha^{-N} & \beta^{-N} \\ -\epsilon(p)\alpha^{-(N+1)} & - \epsilon(p)\beta^{-(N+1)}}.$$
\end{lemma}
\begin{proof}From the definitions.
\end{proof}
\begin{proposition}As functions converging on the open unit disc, we have$$L_p(f,\alpha,\omega^i,T)=\displaystyle \lim_{n\rightarrow \infty} \varepsilon_{\omega^i}L_{N,\alpha}^{\sign(\omega^i)}(T).$$
 \end{proposition}
\begin{proof} Approximation by Riemann sums, and decomposition into tame characters. %CHECK/WRITE DETAILS: $\chi$ picks off three things: the sign, the tame part, and the awesome part which is $(1+T)^{\log_\gamma(a)}$.
\end{proof}

%Now we would like to put all this information together and construct our $L_p^\sharp$ and $L_p^\flat$.
%\begin{proposition}The Mazur-Tate elements $(\vartheta_n)_n$ form a queue sequence.
%\end{proposition}
%\begin{proof}$\pi\vartheta_n=\cdots$
%\end{proof}
\begin{corollary} Let $p$ be supersingular. Then both $\alpha$ and $\beta$ have valuation strictly less than one, so we can reconstruct the $p$-adic $L$-functions by the Mazur-Tate elements:
$$\left(L_p(f,\alpha,\omega^i,T),L_p(f,\beta,\omega^i,T)\right)=\lim_{n\rightarrow \infty} (\theta_n(\omega^i,T),\nu\theta_{n-1}(\omega^i,T))\smat{\alpha^{-N} & \beta^{-N}\\ -\epsilon(p)\alpha^{-(N+1)} & - \epsilon(p)\beta^{-(N+1)}}$$
 In the ordinary case, we have $\ord_p(\alpha)=0<1$, so  $$L_p(f,\alpha,\omega^i,T)=\lim_{n\rightarrow \infty} (\theta_n(\omega^i,T),\nu\theta_{n-1}(\omega^i,T))\smat{\alpha^{-N} \\ -\epsilon(p)\alpha^{-(N+1)}}.$$
\end{corollary}
\begin{proof}This follows from Lemma \ref{MTelementscomein}.
\end{proof}
\rm
In Section $4$, we define an explicit $2\times2$ matrix $\Loghat=\smat{\widehat{\log}_\alpha^\sharp(1+T) & \widehat{\log}_\beta^\sharp(1+T)\\ \widehat{\log}_\alpha^\flat(1+T) & \widehat{\log}_\beta^\flat(1+T)}$ that encodes convenient behavior of the Mazur-Tate elements. We prove that the entries are functions convergent on the open unit disc when $p$ is supersingular, and that $\widehat{\log}_\alpha^\sharp(1+T)$ and $ \widehat{\log}_\alpha^\flat(1+T)$ converge on the closed unit disc in the ordinary case. This assertion is the Main Lemma \ref{convergence}. A corollary of the construction, Remark \ref{logremark}, says that the determinant of $\Loghat$ converges and vanishes precisely at $\zeta_{p^n}=1$ for $n\geq0$. The Main Lemma  \ref{convergence} is the key ingredient to proving our main theorem:
\begin{theorem}\label{maintheorem} Fix a tame character $\omega^i$.
\begin{enumerate}[a.]
 \item When $p$ is supersingular, there is a unique vector of two Iwasawa functions
 
$$\Lvec={\left(\widehat{L}_p^\sharp(f,\omega^i,T),\widehat{L}_p^\flat(f,\omega^i,T)\right)\in\Lambda^{\oplus 2}}$$ so that

$$\left(L_p(f,\alpha,\omega^i,T),L_p(f,\beta,\omega^i,T)\right)=\left(\widehat{L}_p^\sharp(f,\omega^i,T),\widehat{L}_p^\flat(f,\omega^i,T)\right)\Loghat.$$

\item When $p$ is ordinary, there is vector $\Lvec=\left(\widehat{L}_p^\sharp(f,\omega^i,T),\widehat{L}_p^\flat(f,\omega^i,T)\right) \in \Lambda^{\oplus 2}$ so that 

$$\Lvec\smat{\widehat{\log}_\alpha^\sharp(1+T)\\\widehat{\log}_\alpha^\flat(1+T)}=L_p(f,\alpha,\omega^i,T)\text{ and }$$

$$ \widehat{\overrightarrow {L}}_p(f,\omega^i,0)\hidari \widehat{\log}_\beta^\sharp(1) \\ \widehat{\log}_\beta^\flat(1) \migi \textrm{is given by the formulas at $\zeta=1$ of Theorem \ref{vishik}, with $\alpha$ replaced by $\beta$.}$$

Once $\Lvec$ is fixed, all (other) such vectors are given by

 $$\Lvec+g(T)T\left(-\widehat{\log}_\alpha^\flat(1+T),\widehat{\log}_\alpha^\sharp(1+T)\right)$$ 
 
 for $g(T)\in\Lambda$. In particular, the value of $\widehat{\overrightarrow{L}}_p(f,\omega^i,0)$ is uniquely determined.

%\item\label{pseudoconstruction} When $p$ is ordinary, choose $L_p(f,\beta,\omega^i,T)\in K[[T]]$ that is $O({\log_p(1+T)})$ and interpolates well. Then there is a \textbf{unique} vector $\left(\widehat{L}_p^\sharp(f,\omega^i,T),\widehat{L}_p^\flat(f,\omega^i,T)\right)\in\left(\Lambda\otimes \Q\right)^{\oplus2}$ so that
%$$\left(L_p(f,\alpha,\omega^i,T),L_p(f,\beta,\omega^i,T)\right)=\left(\widehat{L}_p^\sharp(f,\omega^i,T),\widehat{L}_p^\flat(f,\omega^i,T)\right)\Loghat.$$
% Conversely, the second factor of $\Lvec\Loghat$ for any choice of $\Lvec$ from part b above gives rise to a critical slope $p$-adic $L$-function $L_p(f,\beta,\omega^i,T)\in K[[T]]$ that is $O({\log_p(1+T)})$ and interpolates well. Any two such choices result in critical slope $p$-adic $L$-functions that differ by $g(T)\log_p(1+T)$ for $g(T)\in\Q\otimes\Lambda$.
\end{enumerate}

The analogous statements in parts a and b with objects without the hats hold.
\end{theorem}

\section{A question by Greenberg}\rm To motivate our theorem, we give a quick application. Greenberg conjectured in \cite{greenberg} that
 $L_p(f,\alpha,\omega^i,T)$ and $L_p(f,\beta,\omega^i,T)$ have finitely many common zeros (in the elliptic curve case) when $p$ is supersingular and $i=0$. In this section, we work in the general supersingular case.
\begin{theorem}[(Rohrlich)] $L_p(f,\alpha,\omega^i,T)$ and $L_p(f,\beta,\omega^i,T)$ vanish at finitely many $T=\zeta_{p^n}-1$. \end{theorem}
\begin{proof}Since these functions interpolate well (Theorem \ref{vishik}), this follows from his original theorem \cite{rohrlich}, which guarantees that $L(f,\chi,1)=0$ at finitely many characters of $p$-power order.
\end{proof}

\begin{theorem}$L_p(f,\alpha,\omega^i,T)$ and $L_p(f,\beta,\omega^i,T)$ have finitely many common zeros. In particular, Greenberg's conjecture is true.
\end{theorem}
\begin{proof}When a zero is not a $p$-power root of unity, it is one of the finitely many zeros of $L_p^\sharp(f,\omega^i,T)$ and $L_p^\flat(f,\omega^i,T)$, since $\det\Log_{\alpha,\beta}(1+T)$ doesn't vanish there. For the other zeros, use Rohrlich's theorem.
\end{proof}

\begin{remark} Pollack already found a different proof in the case $a_p=0$ (\cite[Corollary 5.12]{pollack}).
\end{remark}

\section{The logarithm matrix ${\widehat{\Log}_{\alpha,\beta}(1+T)}$}
\subsection{Definition of the matrix $\Log_{\alpha,\beta}(1+T)$} In this section, we construct a matrix $\Log_{\alpha,\beta}(1+T)$ whose entries are functions converging on the open unit disc in the supersingular case. In the ordinary case, its first column converges on the closed unit disc. They directly generalize the four functions $\log_{\alpha/\beta}^{\sharp/\flat}$ from \cite{shuron} and the three functions $\log_p^+, \log_p^-\cdot \alpha,\log_p^-\cdot \beta$ from \cite{pollack}, all of which concern the supersingular case. We also construct a completed version ${\widehat{\Log}_{\alpha,\beta}(1+T)}$.

\definition Let $i\geq 1$. We \textit{complete} the $p^i$th cyclotomic polynomial by putting $$\widehat{\Phi}_{p^i}(1+T):=\Phi_{p^i}(1+T)/(1+T)^{\frac{1}{2}p^{i-1}(p-1)},$$ except when $p=2$ and $i=1$: To avoid branch cuts (square roots), we set $$\widehat{\Phi}_2(1+T):=\Phi_2(1+T).$$
%\end{definition}

\definition Define the following matrices: $$\CCC_i:=\CCC_i(1+T):=\links a_p & 1 \\ -\epsilon(p)\Phi_{p^i}(1+T) & 0 \rechts, \text{ and } C:=\CCC_i(1)=\links a_p & 1 \\ -\epsilon(p)p & 0\rechts$$ %\end{definition}

\definition\label{logarithmmatrixdefinition} Recall that $N=n+1$ if $p$ is odd, and $N=n+2$ if $p=2$. The logarithm matrix is
$$\Log_{\alpha,\beta}(1+T):=\lim_{n \rightarrow \infty}\CCC_1\cdots \CCC_n C^{-(N+1)}\smat{-1 & -1\\  \beta & \alpha}.$$
%\end{definition}

\begin{remark}[Naming of the matrix $\Log_{\alpha,\beta}$]\label{logremark}For odd $p$, we have:  $$\det\Log_{\alpha,\beta}(1+T)=\frac{\log_p(1+T)}{T}\times\frac{\beta-\alpha}{\left(\epsilon(p)p\right)^2}$$
For $p=2$, the above exponent of $2$ has to be replaced by a $3$. \end{remark}

\begin{convention}
Whenever we encounter an \textit{expression} involving $\Phi_{p^i}(1+T)$, we let $\widehat{expression}$ be the corresponding expression involving $\widehat{\Phi}_{p^i}(1+T)$. For example, we let $\widehat{\CCC}_i$ be $\CCC_i$ with $-\epsilon(p)\widehat{\Phi}_{p^i}(1+T)$ in the lower left entry instead of $-\epsilon(p)\Phi_{p^i}(1+T)$, and
$${\widehat{\Log}_{\alpha,\beta}(1+T)}=\lim_{n \rightarrow \infty}\widehat{\CCC}_1\cdots \widehat{\CCC}_n C^{-(N+1)}\smat{-1 & -1\\  \beta & \alpha}.$$
\end{convention}

\begin{observation}\label{theobservation}
For $i > n\geq 0$, we have $\widehat{\CCC}_i(\zeta_{p^n})=\CCC_i(\zeta_{p^n})=C.$
\end{observation}

\subsection{Convergence of the entries}
\definition Recall that $\ord_p$ is the valuation on $\C_p$ normalized by $\ord_p(p)=1$. Put $$v=\ord_p(a_p) \text{ and } w=\ord_p(\alpha).$$
%\end{definition}
\begin{lemma}[(Main Lemma)]\label{convergence} The entries in the left column of $\Log_{\alpha,\beta}(1+T)$ and ${\widehat{\Log}_{\alpha,\beta}(1+T)}$ are well-defined as power series and converge on the open unit disc. When $v>0$ (the supersingular case) or ${T=\zeta_{p^n}-1}$ with $n\geq 0$, we can say the same about all entries.
\end{lemma}
%This is a generalization of the Convergence Lemma (\cite[Lemma 4.4]{shuron}), and the same proof with some modifications works. We sketch the methods here for completeness and since we need them later on.

 \definition
For a matrix $M=(m_{i,j})_{i,j}$ with entries $m_{i,j}$ in the domain of a valuation $\val$, let the \textit{valuation matrix $[M]$ of $M$} be the matrix consisting of the valuations of the entries:
$$[M]:=[\val(m_{i,j})]_{i,j}$$
 Let $N=(n_{j,k})_{j,k}$ be another matrix so that we can form the product $MN$. Valuation matrices have the following \textit{valuative multiplication} operation:
 $$[M][N]:=[\min_j(m_{i,j}+n_{j,k})]_{i,k}$$
 We also define the \textit{valuation $\val(M)$ of $M$} to be the minimum of the entries in the valuation matrix:
$$\val(M):=\min\{\val(m_{i,j})\}$$
 %\end{definition}
\definition
Let $0<r<1$. Denote by $|\cdot|_p=p^{-\ord_p(\cdot)}$ the $p$-adic absolute value. For $f(T)\in \C_p[[T]]$ convergent on the open unit disc, we define its \textit{valuation at $r$} to be
$$v_r(f(T)):=\inf_{|z|_p<r}\ord_p(f(z)). $$
 We define the \textit{valuation at $0$} to be
$$v_0(f(T)):=\ord_p(f(0)). $$

 %\end{definition}
 \begin{lemma}\label{matrixtriangle}Let $\val$ be a valuation, and $M$ and $N$ be matrices as above allowing a matrix product $MN$. Then $\val(MN)\geq \val(M)+\val(N)$.
 \end{lemma}
 \begin{proof}Term by term, the valuations of the entries of $[MN]$ are at least as big as those of $[M][N]$.
 \end{proof}
\begin{notation}Let $M$ be a matrix whose coefficients are in $\C_p[[T]]$. With respect to $v_r$ we may then define the \textit{valuation matrix of $M$ at $r$} and denote it by $[M]_r$. We similarly define the \textit{valuation of $M$ at $r$} and denote it by $v_r(M)$. When these terms don't depend on $r$ (e.g. when the entries of $M$ are constants), we drop the subscript $r$.
\end{notation}

\begin{example}\label{basicexample}\footnote{This essentially appears in \cite[lemma 4.5]{pollack}. It seems that he meant to write ${p^{-v_r(\Phi_n(1+T))}\sim r^{p^{n-1}(p-1)}}$ in the proof.} Denote the logarithm with base $p$ by $\log_{(p)}$ to distinguish it from the $p$-adic logarithm $\log_p$ of Iwasawa.$$v_r\left(\Phi_{p^n}(1+T)\right)=\begin{cases}1 & \text{when }   r\leq p^{-\frac{1}{p^{n-1}(p-1)}}\\ -\log_{(p)}(r){p^{n-1}(p-1)}&\text{when }  r \geq p^{-\frac{1}{p^{n-1}(p-1)}}\end{cases}$$
\end{example}

\begin{example}\label{normoneexample}$v_r\left((1+T)^{\frac{1}{2}p^{n-1}(p-1)}\right)={\frac{1}{2}p^{n-1}(p-1)}v_r((1+T))=0$.
\end{example}

In what follows, we give the arguments needed for our Main Lemma \ref{convergence} for ${\widehat{\Log}_{\alpha,\beta}(1+T)}$. From Example \ref{normoneexample}, the proof for $\Log_{\alpha,\beta}(1+T)$ follows by taking the hat off the relevant expressions.

\definition\label{many} Assume $\beta\neq\alpha$. We put $\rho:=\frac{\alpha}{\beta}$ and let $\widehat{\Upsilon}_n:=\frac{1}{\beta-\alpha}\left(\beta-\frac{\widehat{\Phi}_{p^n}(1+T)}{\alpha}\right)$. We also put:
$$H_n:=\smat{-1 & -\rho^{n+1}\\ \rho^{-n} & \rho},$$
$$\widehat{M}_n:=\smat{1 & 0\\0 & 1}+H_n\widehat{\Upsilon}_n=\smat{\alpha^n & 0\\0 & \beta^n}\roots^{-1}\widehat{\CCC}_n\roots\smat{\alpha^{-n-1} & 0 \\ 0 & \beta^{-n-1}},$$
$$H_{a,n_1,n_2,\dots,n_l}:=H_aH_{a+n_1+1}H_{a+n_1+n_2+2}\cdots H_{a+n_1+n_2+\cdots+n_l+l},$$
$$\widehat{\Upsilon}_{a,n_1,n_2,\dots,n_l}:=\widehat{\Upsilon}_a\widehat{\Upsilon}_{a+n_1+1}\widehat{\Upsilon}_{a+n_1+n_2+2}\cdots \widehat{\Upsilon}_{a+n_1+n_2+\cdots+n_l+l}.$$

%\end{definition}

Note that $\Log_{\alpha,\beta}$ differs from $\lim_{n\rightarrow\infty}\widehat{M}_1 \widehat{M}_2\cdots \widehat{M}_n$ by multiplication by $\roots$ on the left and a diagonal matrix on the right. To prove convergence, we use an explicit expansion:

\begin{lemma}[Expansion Lemma]\label{zero}
$$\widehat{M}_1\cdots \widehat{M}_n=\smat{1 & 0 \\ 0 & 1}+\sum_{a\geq 1}^n H_a\widehat{\Upsilon}_a + \sum_{\begin{subarray}{c}a\geq1,n_i\geq 1, l\geq 1\\\text{ so that }l+a+\sum_{i\geq 1}^l n_i \leq n \end{subarray}}H_{a,n_1,n_2,\dots,n_l}\widehat{\Upsilon}_{a,n_1,n_2,\dots,n_l}$$\end{lemma}

\begin{proof}
For $n\leq 3$, note that $H_aH_{a+1}=\smat{-1 & -\rho^{a+1} \\ \rho^{-a} & \rho}\smat{-1&-\rho^{a+2}\\\rho^{-a-1} & \rho}=\smat{0&0\\0&0}$. For $n\geq4$, use induction.
\end{proof}

\begin{lemma}\label{one}$H_{a,n_1,n_2,\dots,n_l}=H_a\smat{(-1)^l(1-\rho^{-n_1})(1-\rho^{-n_2})\cdots(1-\rho^{-n_l}) & 0 \\ 0 & \rho^l(1-\rho^{n_1})(1-\rho^{n_2})\cdots(1-\rho^{n_l})}$
\end{lemma}
\proof This follows from the fact that $H_aH_{a+b+1}=H_a\smat{-(1-\rho^{-b}) & 0 \\ 0 & \rho (1-\rho^b)}$.

\begin{proof}[Proof of the Main Lemma \ref{convergence}] We want to prove that the sums involved in Lemma \ref{zero} converge as $n \rightarrow \infty$.

When $\beta=\alpha$, $\ord_p(\alpha)=\frac{1}{2}$, so that the arguments of \cite[Lemma 4.4]{shuron} work, cf. \cite[Remark 5.26]{llz}. Thus, suppose $\beta\neq\alpha$. Fix $r<1$. For the first sum, the terms in the valuation matrix $[H_a\widehat{\Upsilon}_a]_r$ are (up to a constant independent from $a$ or $n$) bounded below by the terms in $\mat{ a & 2wa \\ (2-2w)a & a}$, because $\Pi_{i\geq 1}^{a-1}\Phi_{p^i}(1+T)$ divides $\widehat{\Upsilon}_a$, and in view of Example \ref{basicexample}. All terms go to $\infty$ as $a$ does, except for the upper-right term when $w=0$, or the lower left term when $w=\frac{1}{2}$.

Now we have $\ord_p(\rho(1-\rho^m)) \geq (2w-1)(1+m) $ so that 
$$\begin{array}{ll}[\widehat{\Upsilon}_{a+m}\rho(1-\rho^m)]_r & \geq 2wm+a+C\end{array}$$

for some constant $C$ independent of $a$ or $m$. The easier fact that $[\widehat{\Upsilon}_{a+m}\rho(1-\rho^{-m})]_r\geq a+m$ then allows us to conclude in the supersingular case that all entries in the valuation matrix of $H_{a,m_1,m_2,\dots,n_l}\widehat{\Upsilon}_{a,n_1,n_2,\dots,n_l}$ except possibly for the lower left term have terms bounded below by the corresponding entries of of $H_a\widehat{\Upsilon}_a$, via Lemma \ref{one}.

Thus, three of the terms of $\widehat{\Log}_{\alpha,\beta}(1+T)$ converge in the supersingular case. Since $\det\widehat{\Log}_{\alpha,\beta}(1+T)$ converges as well, all terms of $\widehat{\Log}_{\alpha,\beta}(1+T)$ converge. For ${\Log}_{\alpha,\beta}(1+T),$ we take our hats off.

For the ordinary case, analogous arguments hold for the terms in the left column.\end{proof}

%Fix $l\geq 1$. The valuation matrix for the term $H_{a,m_1,m_2,\dots,n_l}\Upsilon_{a,n_1,n_2,\dots,n_l}$ has terms bounded below by the terms $$\mat{(l+1)a+ln_1+(l-1)n_2+\cdots n_l+\frac{l(l+1)}{2} & la+(l-1)n_1+(l-2)n_2+\cdots n_{l-1}+\frac{l(l+1)}{2}-1 \\ (l+2)a+ln_1+(l-1)n_2+\cdots n_l+\frac{l(l+1)}{2} & (l+1)a+(l-1)n_1+(l-2)n_2+\cdots n_{l-1}+\frac{l(l+1)}{2}-1}$$ y Lemma \ref{one}. Thus, the terms go to $\infty$.

\subsection{The rate of growth}

\definition\label{bigO}For $f(T), g(T)\in \C_p[[T]]$ converging on the open unit disc, we say $f(T)$ is $O(g(T))$ if $$p^{-v_r(f(T))}\text{  is }O(p^{-v_r(g(T))})\text{  as }r\rightarrow 1^-,$$
\text{ i.e. there is an $r_0<1$ and a constant $C$ so that }$$v_r(g(T))<v_r(f(T))+C\text{ when $1>r>r_0$.}$$

If $f(T) $ is $ O(g(T))$ and $g(T) $ is $ O(f(T))$, we say ``$f(T)$ grows like $g(T)$,'' and write $f(T)\sim g(T)$.
%\end{definition}

\begin{example}\label{growthexample}$1\sim T \sim \Phi_p(1+T)$. Also, $\det \Log_{\alpha,\beta}\sim \log_p(1+T)$ by Remark \ref{logremark}.
\end{example}

\begin{proposition}[(Growth Lemma)]\label{growthlemma} When $v>\frac{1}{2}$, the entries of ${\widehat{\Log}_{\alpha,\beta}(1+T)}$ and ${\Log_{\alpha,\beta}(1+T)}$ grow like $\log_p({1+T})^{\frac{1}{2}}$. When $v\leq\frac{1}{2},$ the entries in the left column are $O(\log_p({1+T})^v)$, and those in the right $O(\log_p({1+T})^{1-v})$.
\end{proposition}
We give the proof for $\Log_{\alpha,\beta}(1+T)$, since it is similar for the case ${\widehat{\Log}_{\alpha,\beta}(1+T)}$. Before beginning with the proof, let us name the quantities from Example \ref{basicexample}:
\definition $e_{n,r}:=v_r(\Phi_{p^n}(1+T))=\min(1,-\log_{(p)}(r)(p^n-p^{n-1}))$

\lemma\label{bigolemma} The entries of $\Log_{\alpha,\beta}(1+T)$ are $O(\log_p(1+T)^{1-w})$.

\begin{proof} It suffices to prove this for $\lim_{n\rightarrow \infty} M_1\cdots M_n$, where $M_i$ are as in Definition \ref{many}. Note that 

$$M_n=\Phi_{p^n}(1+T)\smat{\frac{1}{\alpha} & \frac{1}{\beta} \\ \frac{-1}{\alpha} & \frac{-1}{\beta}}+\smat{-\alpha & -\alpha \\ -\beta & -\beta},\text{ so that for $r<1$,}$$
$$[M_n]_r\geq e_{n,r}+w-1\geq (1-w)(e_{n,r}-1).$$
$$\text{ Hence, } [M_1\cdots M_n]_r\geq(1-w)\sum_{i\geq1}^{n}(e_{i,r}-1)=(1-w)\prod_{i\geq 1}^n \left[\frac{\Phi_{p^i}(1+T)}{p}\right]_r.$$ 

Taking limits, the result follows.\end{proof}

\rm We implicitly used diagonalization in an earlier proof, but write it out for convenience:
\observation\label{diagonalization} Let $m$ be an integer. Then $\roots\smat{\alpha^{m} & 0 \\ 0 & \beta^{m}}=C^{m}\roots$.

\begin{proof}[Proof of Growth Lemma \ref{growthlemma}]\rm We first treat the case $v=0=\ord_p(\alpha)$. When $n\geq1$,
$$\left[\CCC_1\cdots\CCC_n\right]_r=[\CCC_1]_r\cdots[\CCC_n]_r=\bai 0 & 0 \\ e_{1,r} & e_{1,r}\dai.$$
By Observation \ref{diagonalization}, the valuation matrix of the left column of $\CCC_1\cdots\CCC_n\smat{ -1 & -1 \\ \beta & \alpha }$ and of $\Log_{\alpha,\beta}$ is $\bai0 \\ e_{1,r}\dai$. Thus, these entries are $O(\Phi_p(1+T))$. Since we have $\Phi_p(1+T)\sim 1$ by Example \ref{growthexample}, they are indeed $O(1)$.

Next, we assume $0<v\leq \frac{1}{2}$. Given $r$, let $i$ be the largest integer so that $e_{i,r}<2v$. Without loss of generality, assume $i$ is even. We then compute
$$\left[\CCC_1\cdots\CCC_i\right]_r=[\CCC_1]_r\cdots[\CCC_i]_r=\bai e_{2,r}+e_{4,r}+\cdots +e_{i,r} & v+ e_{2,r}+\cdots +e_{i-2,r}\\v+e_{1,r}+e_{3,r}+\cdots +e_{i-1,r} & e_{1,r}+\cdots+e_{i-1,r}\dai.$$
Remembering that $e_{i+1,r}\geq 2v$, we see that for $n> i$,
$$[\CCC_1\cdots\CCC_n]_r=\bai (n-i)v+e_{2,r}+e_{4,r}+\cdots +e_{i,r} & (n-i-1)v+e_{2,r}+\cdots +e_{i,r}\\ \geq (n-i-1)v+e_{1,r}+\cdots +e_{i-1,r} & \geq(n-i)v+e_{1,r}+\cdots +e_{i-1,r}\dai,$$
where by $\geq x$ we have denoted an unspecified entry that is greater than or equal to $x$. By Observation \ref{diagonalization}, we have that
$\left[\CCC_1\cdots\CCC_nC^{-(N+1)}\roots\right]_r$ is  $$\bai \geq (n-N-i)v+e_{2,r}+\cdots +e_{i,r} & \geq (n+N-i)v-N+e_{2,r}+\cdots +e_{i,r} \\ \geq(n-N-i+1)v+e_{1,r}+\cdots +e_{i-1,r} & \geq(n+N-i+1)v-N+e_{1,r}+\cdots+e_{i-1,r}\dai.$$

Now let $m:=i-\lfloor\ord_p(2v)\rfloor$. We then have $e_{m-h,r}\cdot 2v \leq e_{i-h,r}$ for any $h<i$, and $e_{M,r}=1$ for any $M>m$. Thus, the top left entry of $\left[\CCC_1\cdots\CCC_nC^{-(N+1)}\roots\right]_r$ is up to a constant independent from $r$ greater than or equal to
$$2v(e_{m-i+2,r}+e_{m-i+4,r}+\cdots + e_{m,r})-iv=2v\cdot v_r\left(\prod_{k\geq m-i+2, \text{$k$ is even}}^m \frac{\Phi_{p^k}(1+T)}{p}\right).$$
Now note that $\prod_{\text{even }k\geq 2}^\infty \frac{\Phi_{p^k}(1+T)}{p}\sim \log_p(1+T)^{\frac{1}{2}}$. Using similar arguments, one obtains the appropriate bound for the lower left entry. \footnote{The same arguments show it for the right entries when $v=\frac{1}{2}$, although this already follows from Lemma \ref{bigolemma}.} The claim for the case $0\leq v\leq\frac{1}{2}$ thus follows from Lemma \ref{bigolemma}.

Lastly, we treat the case $v>\frac{1}{2}$. Without loss of generality, let $n>1$ be even. Then
$$\left[\CCC_1\cdots\CCC_n\right]_r=[\CCC_1]_r\cdots[\CCC_n]_r=\bai e_{2,r}+e_{4,r}+\cdots+e_{n,r} & v+e_{1,r}+\cdots+ e_{n-2,r} \\ v+e_{1,r}+\cdots+e_{n-1,r} & e_{1,r}+\cdots+e_{n-1,r}\dai.$$

From Observation \ref{diagonalization}, $\ord_p(\alpha)=\ord_p(\beta)=\frac{1}{2}$, and $e_{n,r}\leq1,$ we compute
$$\left[\CCC_1\cdots\CCC_nC^{-(N+1)}\roots\right]_r=\bai-\frac{N}{2}+e_{2,r}+\cdots+e_{n,r} & -\frac{N}{2}+e_{2,r}+\cdots+e_{n,r} \\ \frac{1-N}{2}+e_{1,r}+\cdots+e_{n-1,r} & \frac{1-N}{2}+e_{1,r}+\cdots+e_{n-1,r}\dai.$$
Up to a constant independent from $r$, these entries are
$$v_r\left(\prod_{k\geq 2, k \text{ even}}^n \frac{\Phi_{p^k}(1+T)}{p}\right)=e_{2,r}+\cdots+e_{n,r}-\frac{n}{2} \text{, and }$$$$ v_r\left(\prod_{k\geq 1, k \text{ odd}}^n \frac{\Phi_{p^k}(1+T)}{p}\right)=e_{1,r}+\cdots+e_{n-1,r}-\frac{n}{2}. $$

But from \cite[Lemma 4.5]{pollack}, we have $$\prod_{\text{even }k\geq 2}^\infty \frac{\Phi_{p^k}(1+T)}{p}\sim \prod_{\text{odd }k\geq 1}^\infty \frac{\Phi_{p^k}(1+T)}{p}\sim \log_p(1+T)^{\frac{1}{2}},$$ from which the assertion follows for the case $v>\frac{1}{2}$.
\end{proof}

\begin{lemma}\label{505}When $v=0$, the functions $\widehat{\log}_\alpha^\sharp(1+T)$ and $\widehat{\log}_\alpha^\flat(1+T)$ are in $\Lambda$.
\end{lemma}
\begin{proof}
 Observation \ref{diagonalization} and Definition \ref{logarithmmatrixdefinition}.
\end{proof}

\subsection{The functional equation}

\begin{proposition}\label{functionalequation}
Under the change of variable $(1+T)\mapsto(1+T)^{-1}$, the first column of $\Loghat$ is invariant. When all four entries of $\Loghat$ converge, then:

If $p$ is odd, then ${\widehat{\Log}_{\alpha,\beta}(1+T)}=\widehat{\Log}_{\alpha,\beta}((1+T)^{-1})$.

If $p=2$, then $\smat{1 & 0 \\ 0 & (1+T)^{-1}}{\widehat{\Log}_{\alpha,\beta}(1+T)}=\widehat{\Log}_{\alpha,\beta}((1+T)^{-1})$.\end{proposition}
\proof \rm All $\widehat{\CCC}_i(1+T)$ are invariant under the change of variables $1+T\mapsto \frac{1}{1+T}$, except $\widehat{\CCC}_1(1+T)$ if $p=2$, where we have $\widehat{\CCC}_1(\frac{1}{1+T})=\smat{1 & 0 \\ 0 & (1+T)^{-1}}\widehat{\CCC}_1(1+T)$.

\subsection{The functional equation in the case $a_p=0$.}
\rm
 When $a_p=0$, the entries of ${\widehat{\Log}_{\alpha,\beta}(1+T)}$ are off by units from the corresponding ones in $\Log_{\alpha,\beta}(1+T)$.
More precisely, denote by $\log_p^{\pm}(1+T)$ Pollack's half-logarithms \cite{pollack}:
$$\log_p^+(T):=\frac{1}{p}\prod_{j\geq 1}\frac{\Phi_{p^{2j}}(1+T)}{p},\text{ and }\log_p^-(T):=\frac{1}{p}\prod_{j\geq 1}\frac{\Phi_{p^{2j-1}}(1+T)}{p}.$$

We then have $$\Log_{\alpha,\beta}(1+T)=\begin{cases}\frac{1}{\epsilon(p)}\smat{ \log_p^+(T) & \log_p^+(T)\\ \log_p^-(T)\alpha & \log_p^-(T)\beta }& \text{when $p$ is odd,}\\\frac{1}{\epsilon(2)}\smat{ \frac{-1}{\epsilon(2)2}\log_2^+(T)\alpha & \frac{-1}{\epsilon(2)2}\log_2^+(T)\beta\\ \log_2^-(T) & \log_2^-(T)}& \text{when $p=2$.}\end{cases}$$
Setting $U^\pm(1+T):=\widehat{\log_p^\pm(T)}/\log_p^\pm(T)$, we obtain
$${\widehat{\Log}_{\alpha,\beta}(1+T)}=\smat{U^+(1+T) & 0 \\ 0 & U^-(1+T)}\Log_{\alpha,\beta}(1+T).$$
Now put $$W^+(1+T)=\frac{U^+(1+T)}{U^+((1+T)^{-1})}=\prod_{j\geq 1}(1+T)^{-p^{2j-1}(p-1)},\text{ and }$$
$$\ W^-(1+T)=\begin{cases}\frac{U^-(1+T)}{U^-((1+T)^{-1})} = \prod_{j\geq 1}(1+T)^{-p^{2j-2}(p-1)}&\text{for odd $p$,}\\\frac{U^-(1+T)}{(1+T)U^-((1+T)^{-1})}= (1+T)^{-1}\prod_{j\geq 2}(1+T)^{-p^{2j-2}(p-1)} &\text{when $p=2$.}\end{cases}$$

We can finally arrive at the corrected statement of \cite[Lemma 4.6]{pollack}:

\begin{lemma}\label{functionalequationforlog}We have
$$\log_p^+(T)W^+(1+T)=\log_p^+\left(\frac{1}{1+T}-1\right),$$$$
\log_p^-(T)W^-(1+T)=\log_p^-\left(\frac{1}{1+T}-1\right).$$
\end{lemma}
\begin{proof} This follows from what has been said above, or by going through the proof of \cite[Lemma 4.6]{pollack} on noting that the units $U^\pm(1+T)\neq 1$.
\end{proof}

\section{The two $p$-adic $L$-functions $\widehat{L}_p^\sharp(f,T)$ and $\widehat{L}_p^\flat(f,T)$}
In this section, we construct Iwasawa functions $\widehat{L}_p^\sharp(f,T)$ and $\widehat{L}_p^\flat(f,T)$. We present the arguments with the completions. The corresponding non-completed arguments can be recovered by taking off the hat above any expression $\widehat{xyz}$ and replacing it by $xyz$.
Instead of working with the matrices $\widehat{\CCC}_i$ and $C$, we make our calculations easier via the following definitions:
\definition We put

$$\widehat{\CC}_i:= \links a_p & \widehat{\Phi}_{p^i}(1+T) \\ -\epsilon(p) & 0 \rechts, A:=\links a_p & p \\ -\epsilon(p) & 0\rechts, \tilde{A}:=\links a_p & 1 \\ -\epsilon(p) & 0\rechts.$$

%\end{definition}

\definition For any integer $i$, put $Y_{2i}:=p^{-i}A^{2i}, $ and $Y_{2i+1}=Y_{2i}\tilde{A}$.
%\end{definition}

\begin{proposition}[(Tandem Lemma)]\label{tandemlemma}Fix $n\in\N$. Assume that for any $i \in \N$, we are given functions $Q_i=Q_i(T)$ so that $Q_i\in\Phi_{p^i}(1+T)\O[T]$ whenever $i \leq n$, and\\ $(Q_{n+1}, Q_n)Y_{n'-n}=(Q_{n'+1},Q_{n'})$ for any $n'\in\N$. Then
$$(Q_{n+1},Q_{n})=(\widetilde{q_1},q_0)\widehat{\CC}_1\cdots \widehat{\CC}_n \text{ with }\widetilde{q_1},q_0\in\O[T].$$
\end{proposition}
\begin{proof}We inductively show that $(Q_{n+1},Q_n)=(\widetilde{q}_{i+1},q_i)\widehat{\CC}_{i+1}\cdots \widehat{\CC}_n$ for $\widetilde{q}_{i+1},q_i \in \O[T]$ with ${0\leq i\leq n}$: Note that at the base step $i=n$, the product of the $\widehat{\CC}$'s is empty so that we indeed have $(\widetilde{q}_{n+1},q_n)=(Q_{n+1},Q_n).$ For the inductive step, let $i\geq 1$. Then we have
$$(Q_{n+1},Q_n)=(\widetilde{q}_{i+1},q_i)A^{n-i} \text{ by evaluation at $\zeta_{p^i}-1$, and }$$ $$(Q_{n+1}, Q_n)Y_{i-n}=(Q_{i+1}, Q_i)\text{ by assumption}.$$
We thus have $(\widetilde{q}_{i+1},q_i)A^{n-i}Y_{i-n}=(Q_{i+1},0)$ at $\zeta_{p^i}-1$, whence $q_i$ vanishes at $\zeta_{p^i}-1$. We hence write $q_i=\widehat{\Phi}_{p^i}(1+T) \cdot \widetilde{q_i}$ for some $\widetilde{q_i}\in\O[T]$. Now put ${(\widetilde{q_i},q_{i-1}):=(\widetilde{q}_{i+1},\widetilde{q_i})\tilde{A}^{-1}}$. Then
$(\widetilde{q}_{i+1},q_i)=(\widetilde{q_i},q_{i-1})\widehat{\CC}_i.$
\end{proof}

\begin{observation}Let $(\Theta_n)_n$ be a queue sequence and $\pi:\Lambda_n\rightarrow \Lambda_{n-1}$ be the projection. Then for $n\geq2$, we have $\pi(\Theta_n,\nu\Theta_{n-1})=(\Theta_{n-1},\nu\Theta_{n-2})A$.
\end{observation}
\begin{proof}Definition \ref{queuesequence}.
\end{proof}
\begin{proposition}\label{zerofindinglemma}Let $(\Theta_n)_n$ be a queue sequence and $0\leq n'\leq n$. When lifting elements of $\Lambda_n$ to $\O[T]$, the second entry of $(\Theta_n,\nu\Theta_{n-1})Y_{n'-n}$ vanishes at $\zeta_{p^{n'}}-1$.
\end{proposition}
\begin{proof}Denote by $\pi_{n/n'}$ the projection from $\Lambda_n$ to $\Lambda_{n'}$. By the above observation, the second entry of $$\pi_{n/n'}(\Theta_n,\nu\Theta_{n-1})Y_{n'-n}=(\Theta_{n'},\nu\Theta_{n'-1})A^{n-n'}Y_{n'-n}$$ is contained in the ideal $(\Phi_{n'})\subset\Lambda_{n'}$. Thus, its preimage under $\pi_{n/n'}$ is in the ideal $(\Phi_{n'})\subset\Lambda_{n}$.
\end{proof}

\begin{corollary}\label{endgame} Let $(\Theta_n)_n$ be a queue sequence. Then $(\Theta_n,\nu\Theta_{n-1})=\widehat{\Upsilon}_n \widehat{\CCC}_1\cdots\widehat{\CCC}_n\tilde{A}^{-1}$ for some $\widehat{\Upsilon}_n\in\Lambda_n^{\oplus 2}$.
\end{corollary}
\begin{proof}We identify elements of $\Lambda_n$ by their corresponding representative in $\O[T]$ and use Proposition \ref{zerofindinglemma}. Then, we can apply the Tandem Lemma  \ref{tandemlemma}, and project back to $\Lambda_n^{\oplus2}$.
\end{proof}

\begin{corollary}\label{awesome} We rewrite the Riemann sum approximations of Definition \ref{riemannsumapproximation}: For some $\widehat{\overrightarrow{L^{\omega^i}_{p,n}}}\in \O[T]^{\oplus 2},$$$\begin{array}{ll}\left(\varepsilon_{\omega^i}L_{N,\alpha}^{sign(\omega^i)},\varepsilon_{\omega^i}L_{N,\beta}^{sign(\omega^i)}\right)  & =\widehat{\overrightarrow{L^{\omega^i}_{p,n}}}\widehat{\CCC}_1\cdots\widehat{\CCC}_n\tilde{A}^{-1}\smat{ \alpha^{-N} & \beta^{-N} \\ -\alpha^{-(N+1)} & -\beta^{-(N+1)}} \\& = \widehat{\overrightarrow{L^{\omega^i}_{p,n}}}\widehat{\CCC}_1\cdots\widehat{\CCC}_nC^{-(N+1)}\smat{ -1 & -1 \\ \beta & \alpha}.\end{array}$$ \end{corollary}

\begin{proof}We know that $(\alpha^{N+1}L_{N,\alpha},\beta^{N+1}L_{N,\beta})=(\vartheta_N,\nu\vartheta_{N-1})\smat{\alpha & \beta\\ -\epsilon & -\epsilon}$. The isotypical components of $\vartheta_N$ form queue sequences. Now apply Corollary \ref{endgame} and lift back to $\O[T]^{\oplus 2}$.  \end{proof} %In view of proposition \ref{zerofindinglemma}, we can apply the Tandem Lemma \ref{tandemlemma} to their lifts in $\O[T]$.
%Taking $\displaystyle\lim_{n\rightarrow \infty}$,
The above $\widehat{\overrightarrow{L^{\omega^i}_{p,n}}}$ are not unique, so we take limits by regarding the polynomials as elements of $\Lambda_n^{\oplus 2}$:
\definition We define $\widehat{\overrightarrow{L^{\omega^i}_{p}}}$ as follows. 
$$\widehat{\overrightarrow{L^{\omega^i}_{p}}} :=  \lim_{n\rightarrow \infty}\widehat{\overrightarrow{L^{\omega^i}_{p,n}}}\in\Lambda^{\oplus2}/\gM,$$

%=\left(\widehat{L}_p^\sharp(f,\omega^i,T),\widehat{L}_p^\flat(f,\omega^i,T)\right)

where $\gM$ is defined as follows:

\definition We put $\gM:=\varprojlim_n \gM_n$, where

$$\gM_n:=\ker\left(\times \widehat{\CCC}_1\cdots\widehat{\CCC}_nC^{-(N+1)}\roots\right)\subset\Lambda_n\oplus \Lambda_n.$$

\proposition \label{yeah}
For supersingular $p$, $\gM$ is trivial. For ordinary $p$, $\gM\cong T\Lambda\left(-\log_\alpha^\flat \oplus \log_\alpha^\sharp\right)\subset\Lambda\oplus\Lambda$.

\begin{proof}
 Since $C^{-(N+1)}\roots=\roots\smat{-\alpha^{-(N+1)} & 0 \\ 0 & \beta^{-(N+1)}}$, we have 
 
 $$\gM_n = p^{\ord_p(\alpha)(N+1)}\left(\alpha^{N+1}\Lambda_n\oplus \beta^{N+1}\Lambda_n\right)\smat{\alpha & 1 \\ -\beta & -1}\widehat{\CCC}_n^*\cdots \widehat{\CCC}_1^*(\beta-\alpha)T,$$
 
where $\widehat{\CCC}_i^*$ is the adjugate of $\widehat{\CCC}_i$ (cf. also Lemma 5.8 in \cite{shuron}).

Since the matrix product to the right of $(\alpha^{N+1}\Lambda_n\oplus \beta^{N+1}\Lambda_n)$ has $\Lambda_n$-integral coefficients, we see that $\gM_n\subset p^{\ord_p(\alpha)(N+1)}\Lambda_n^{\oplus 2}$ so that $\varprojlim_n \gM_n=0$ when $\ord_p(\alpha)>0$. In the ordinary case, only the terms involving a power of $\beta$ go to zero in the limit, whence the result.
\end{proof}

\begin{proof}[\textit{Proof of Theorem \ref{maintheorem}}]: We give the proof with the hats, since the proof for the expressions without the hats is the same. Part a follows from taking limits of $\widehat{\overrightarrow{L^{\omega^i}_{p,n}}}$ together with the Main Lemma \ref{convergence} and the above Proposition \ref{yeah} (triviality of $\gM$). For part b, the proof is the same up to the description of $\gM$ and Proposition \ref{yeah}, which gives rise to the term $g(T)T\left(-\widehat{\log}_\alpha^\flat\oplus\widehat{\log}_\alpha^\sharp\right)$. Now use Lemma \ref{505}. \end{proof}\rm

Now that we have finally proved Theorem \ref{maintheorem}, we can give the following corollary:

\begin{corollary}Pick $T$ so $\Log_{\alpha,\beta}(1+T)$ and $\widehat{\Log}_{\alpha,\beta}(1+T)$ converge in all entries and are invertible. Then $${(\widehat{L}_p^\sharp(f,\omega^i,T),\widehat{L}_p^\flat(f,\omega^i,T))=(L_p^\sharp(f,\omega^i,T),L_p^\flat(f,\omega^i,T))\Log_{\alpha,\beta}(1+T){\widehat{\Log}_{\alpha,\beta}(1+T)}^{-1}}.$$
\end{corollary}

\begin{remark}In our setup so far, we have worked with the periods $\Omega_f^\pm$. In the case of an elliptic curve $E$ over $\Q$, one can alternatively use the real and imaginary N\'{e}ron periods $\Omega_E^\pm$. These real and imaginary N\'{e}ron periods are defined as follows:\end{remark}

\definition\label{neronperiods}Decompose $H_1(E,\R)=H_1(E,\R)^+\oplus H_1(E,\R)^-$, where complex conjugation acts as $+1$ on the first summand and as $-1$ on the second. Put $H_1^\pm(E,\Z):=H_1(E,\Z)^\pm\cap H_1(E,\R)$. Choose generators $\delta^\pm$ of $H_1(E,\Z)^\pm$ so that the following integrals are positive:
$$\Omega_E^\pm:=\begin{cases}\int_{\delta^\pm}\omega_E & \text{ if $E(\R)$ is connected,}\\ 2\cdot\int_{\delta^\pm}\omega_E & \text{ if not.} \end{cases}$$
%\end{definition}

\begin{convention} \rm When working with these periods, we may define modular symbols and $p$-adic $L$-functions analogously, and write $E$ wherever we have written $f$ before. \end{convention}

 \rm In view of \cite{BCDT} and \cite{wiles}, we have a modular parametrization $\pi:X_0(N)\rightarrow E$, so that $\pi^*(\omega_E)=c\cdot f_E\cdot \frac{dq}{q}$ for some normalized weight two newform $f_E$ of level $N$. The constant $c$ is called the \textbf{Manin constant} for $\pi$. It is known to be an integer (cf. \cite[Proposition 2]{edixhoven}) and conjectured to be $1$. See \cite[$\S$ 5]{maninconstant}.

We note that the analogue of Theorem \ref{integrality} is not necessarily satisfied when one replaces $\Omega_f^\pm$ by $\Omega_E^\pm$, but the following is known (cf. \cite[Remark 5.4, Remark 5.5]{pollack}):
\begin{theorem}[(Imitation of Theorem \ref{integrality})]\label{technicality}
Let $E$ be a strong Weil curve over $\Q$, and $p$ be a prime of good reduction. Then:
\begin{enumerate}
\item \text{\cite[Th\'{e}or\`{e}me A]{abbesullmo} $p$ does not divide $c$.}
\item \text{\cite[Theorem 3.3]{maninconstant} If $a_p\not\equiv 1\mod{p}$, we have $2\left[\frac{a}{p^n}\right]_E^\pm\in c^{-1}\Z$,}
\text{so $2\left[\frac{a}{p^n}\right]_E^\pm\in \Z_p$.}
\end{enumerate}

\end{theorem}

\begin{corollary}When $a_p\not\equiv 1 \mod{p}$, $L_p^\sharp(E,\omega^i,T)$ and $L_p^\flat(E,\omega^i,T)$ and their completions are in $\Lambda$. In particular, the $2$-adic $L$-functions $L_2^\sharp(E,\omega^i,T)$ and $L_2^\flat(E,\omega^i,T)$ from \cite[Definition 6.1]{shuron} agree with those of this paper and are consequently elements of $\Lambda$, rather than $\Lambda\otimes\Q$. \end{corollary}
\begin{proof}This follows from Theorem \ref{maintheorem} and what has just been said. For $p=2$, we exploit the following symmetry in the isotypical components of the Riemann sums $L_{N,\alpha}^\pm$ and $L_{N,\beta}^\pm$: From $\eta^\pm(\frac{a}{m})=\pm\eta^\pm(\frac{-a}{m})$, we can conclude that $\omega^i(a)\eta^\pm(\frac{a}{m})=\pm\omega^i(-a)\eta^\pm(\frac{-a}{m})$.
\end{proof}

\begin{corollary}[(Analogue of Theorem \ref{maintheorem})]\label{afterthiscomesthetable} When $a_p\not\equiv 1\mod{p}$, the statement of Theorem \ref{maintheorem} with $f$ formally replaced by $E$ is still valid. When ${a_p\equiv 1\mod{p}}$ or $E$ is not a strong Weil curve, we can say the same with the added caveat that $L_p^\sharp(E,\omega^i,T)$, $L_p^\flat(E,\omega^i,T)$, and their completions are elements of $\Q\otimes \Lambda$.
\end{corollary}

From Theorem \ref{vishik}, we can give a table of the special values for a good prime $p$:

$$\begin{tabular}{ccc}

\hline\noalign{\smallskip}
&$L_p^{\sharp}\left(f,{\omega^i},0\right)$ & $L_p^{\flat}\left(f,{\omega^i},0\right)$\\
\noalign{\smallskip}\hline\noalign{\smallskip}
$p$ odd, $i=0$&$\left(-a_p^2+2a_p+p-1\right)\frac{L\left(f,1\right)}{\Omega_f^+}$&$\left(2-a_p\right)\frac{L\left(f,1\right)}{\Omega_f^+}$\\\hline
$p$ odd, $i\neq 0$&$-pa_p\frac{L\left(f,{\omega^{-i}},1\right)}{\tau\left({\omega^{-i}}\right)\Omega_f^{{\omega^i}\left(-1\right)}}$&$-p\frac{L\left(f,{\omega^{-i}},1\right)}{\tau\left({\omega^{-i}}\right)\Omega_f^{{\omega^i}\left(-1\right)}}$\\\hline
$p=2$, $i=0$&$\left(-a_p^3+2a_p^2+2pa_p-a_p-2p\right)\frac{L\left(f,1\right)}{\Omega_f^+}$&$\left(-a_p^2+2a_p+p-1\right)\frac{L\left(f,1\right)}{\Omega_f^+}$\\\hline
$p=2$, $i\neq 0$&$-p^2a_p\frac{L\left(f,{\omega^{-i}},1\right)}{\tau\left({\omega^{-i}}\right)\Omega_f^{{\omega^i}\left(-1\right)}}$&$-p^2\frac{L\left(f,{\omega^{-i}},1\right)}{\tau\left({\omega^{-i}}\right)\Omega_f^{{\omega^i}\left(-1\right)}}$\\\noalign{\smallskip}\hline

\end{tabular}$$

In view of these special values, it seems reasonable to make the following conjecture:
\begin{conjecture} Let $f$ be a modular form as above, and let $p$ be a good supersingular prime. When $p$ is odd, $\widehat{L}_p^\flat(f,\omega^i,T)$ and ${L_p^\flat}(f,\omega^i,T)$ are not identically zero, and $\widehat{L}_p^\sharp(f,\omega^i,T)$ and ${L_p^\sharp}(f,\omega^i,T)$ are not identically zero when $a_p\neq 2$. When $p=2,$ the power series $\widehat{L}_2^\sharp(f,\omega^i,T)$ and ${L_2^\sharp}(f,\omega^i,T)$ are not identically zero, and $\widehat{L}_2^\flat(f,\omega^i,T)$ and ${L_2^\flat}(f,\omega^i,T)$ are not identically zero when $a_2\neq 1$.
\end{conjecture}
\subsection{The functional equation in the supersingular case}
\rm
Let $f$ be a weight two modular form of level $N$ and nebentype $\epsilon$ which is an eigenform for all $T_n$. Recall also Definition \ref{log} of $\log_\gamma(\cdot)$. We denote by $f^*(z)=w_N(f(z))=\epsilon(-1)f(\frac{-1}{Nz})$ the involuted form of $f$ under the Atkin-Lehner/Fricke operator, as in \cite[(5.1)]{mtt}, and let $\alpha^*=\frac{\alpha}{\epsilon(p)}$ and $\beta^*=\frac{\beta}{\epsilon(p)}$.
\begin{theorem}\label{functionalequation}Let $p$ be a supersingular prime so that $(p,N)=1$, i.e. $N\in\Z_p^\times\cong\G_\infty$. %When ${\epsilon(p)=1}$, $\widehat{L}_p^\sharp(f,\omega^i,T)$ and $\widehat{L}_p^\flat(f,\omega^i,T)$ satisfy the following functional equation:
%$$\begin{array}{ll}\widehat{L}_p^\sharp(f,\omega^i,T)=-\epsilon(-1)(1+T)^{-\log_\gamma(N)}\omega^i(-N)\widehat{L}_p^\sharp(f,\omega^i,\frac{1}{1+T}-1),\\
 %\widehat{L}_p^\flat(f,\omega^i,T)=-\epsilon(-1)(1+T)^{-\log_\gamma(N)}\omega^i(-N)\widehat{L}_p^\flat(f,\omega^i,\frac{1}{1+T}-1).\end{array}$$
Then
 $$\begin{array}{l}\left(\widehat{L}_p^\sharp(f,\omega^i,T),\widehat{L}_p^\flat(f,\omega^i,T)\right){\widehat{\Log}_{\alpha,\beta}(1+T)}=\\-\epsilon(-1)\omega^{-i}(-N)(1+T)^{-\log_\gamma(N)}\left(\widehat{L}_p^\sharp(f^*,\omega^{-i},\frac{1}{1+T}-1),\widehat{L}_p^\flat(f^*,\omega^{-i},\frac{1}{1+T}-1)\right)\widehat{\Log}_{\alpha^*,\beta^*}(1+T).\end{array}$$
\end{theorem}

\begin{corollary}\label{functionalequationcorollary}For an elliptic curve $E$ over $\Q$ and a good supersingular prime $p$, let $c_N$ be the sign of $f$, i.e. $f^*:=-c_Nf$ (cf. \cite[\S 18]{mtt}). We then have$$\begin{array}{l}\widehat{L}_p^\sharp(E,\omega^i,T)=-(1+T)^{-\log_\gamma(N)}\omega^i(-N)c_N\widehat{L}_p^\sharp(E,\omega^i,\frac{1}{1+T}-1), \\
 \widehat{L}_p^\flat(E,\omega^i,T)=-(1+T)^{-\log_\gamma(N)}\omega^i(-N)c_N\widehat{L}_p^\flat(E,\omega^i,\frac{1}{1+T}-1).\end{array}$$
When $a_p=0$, we can give an explicit functional equation for the non-completed $p$-adic $L$-functions, which corrects \cite[Theorem 5.13]{pollack} in the case $i=0$:
$$\begin{array}{l}L_p^\sharp(E,\omega^i,T)=-(1+T)^{-\log_\gamma(N)}\omega^i(-N)c_NW^+(1+T)L_p^\sharp(E,\omega^i,\frac{1}{1+T}-1), \\
L_p^\flat(E,\omega^i,T)=-(1+T)^{-\log_\gamma(N)}\omega^i(-N)c_NW^-(1+T)L_p^\flat(E,\omega^i,\frac{1}{1+T}-1).\end{array}$$
\end{corollary}
\begin{proof}[Proof of Theorem \ref{functionalequation}] This follows from the functional equations for $\widehat{L}_p(f,\alpha,\omega^i,T)$ and $\widehat{L}_p(f,\beta,\omega^i,T)$, which formally display exactly the same invariance under the substitution $T\mapsto \frac{1}{1+T}-1$, cf. \cite[\S 17, (17.3)]{mtt}. %The ``root numbers'' in the functional equation come from \cite[Proposition in \S I.6]{mtt}, where we have $$\left[\frac{a}{p^n}\right]^\pm_f=-\epsilon(-p^n)\left[\frac{a'}{p^n}\right]^\pm_{f^*}=\epsilon(p^n)\left[\frac{a'}{p^n}\right]^\pm_{f},$$ where $a'\in\left(\Z/p^n\Z\right)^*$ is chosen so that $a'\equiv \frac{-1}{Na}\mod{p^n}$. 
The rest is invariance of $\Loghat$ under $T\mapsto \frac{1}{1+T}-1$, cf. Proposition \ref{functionalequation}.
\end{proof}

\part{Invariants coming from the conjectures of Birch and Swinnerton-Dyer in the cyclotomic direction}
\section{The conjectures about the rank and leading coefficient}
\rm We scrutinize what happens when $T=\zeta_{p^n}-1$ for $n\geq1$: We estimate BSD-theoretic quantities in the cyclotomic direction, using the pairs of Iwasawa invariants of $L_p^\sharp$ and $L_p^\flat$ (which match those of $\widehat{L}_p^\sharp$ and $\widehat{L}_p^\flat$ when used, cf. Lemma \ref{invariantsarethesame}).

Choose $\G_f\subset\Gal(\overline{\Q}/\Q)$ so that $\{f^\sigma=\sum \sigma(a_n)q^n\}_{\sigma\in\G_f}$ contains each Galois conjugate of $f$ once.

\begin{definition}  For $\sigma\in\G_f,$ the $\sigma$-parts of the ($p$-adic) analytic ranks of $A_f(\Q_n)$ and of $A_f(\Q_\infty)$ are
 
 $$r_n^{an}(f^\sigma) = \sum_{\zeta: \text{ $p^n$th roots of unity}}\ord_{\zeta-1} (L_p(f^\sigma,\alpha,T)) \text{ and }$$ 
 $$r_\infty^{an}(f^\sigma):=\lim_{n\rightarrow \infty}r_n^{an}(f^\sigma)=\sum_{\zeta: \text{\textit{all} $p$-power roots of unity}}\ord_{\zeta-1} (L_p(f^\sigma,\alpha,T)).$$ Note that by a theorem of Rohrlich \cite{rohrlich}, this is a finite integer.\end{definition}
 
We can then estimate the $p$-adic analytic rank of $A_f(\Q_n)$ and of $A_f(\Q_\infty)$ by setting
\begin{equation}\label{sums}r_n^{an}:=\sum_{\sigma\in\G_f} r_n^{an}(f^\sigma) \text{ and } r_\infty^{an}:=\sum_{\sigma\in\G_f} r_\infty^{an}(f^\sigma).\end{equation}
%where $f^\sigma$ goes through the Galois conjugates $f^\sigma=\sum \sigma(a_n)q^n$ for $\sigma\in\Gal(\overline{\Q}/\Q).$  %where we take the sum over all $p^n$ th roots of unity $\zeta$ and $\ord_{\zeta-1}(\cdot)$ represents the order of vanishing at $\zeta-1$. We also put
Conjecturally, $r_\infty^{an}$ should agree with the complex analytic rank of $A_f(\Q_\infty)$ defined by the order of vanishing of the Hasse-Weil series $L(A_f/\Q_\infty,s)$ at $s=1$. 

\begin{definition} We let $d_n$ be the normalized jump in the ranks of $A_f$ at level $\Q_n$:$$d_n:=\frac{\rank(A_f(\Q_n))-\rank(A_f(\Q_{n-1}))}{p^n-p^{n-1}}$$
%$$d_n^{an}:=\frac{\rank^{an} E(\Q_n)-\rank^{an} E(\Q_{n-1})}{p^n-p^{n-1}}$$

Denote by $D(\Q_n)$ the discriminant, by $R(A_f/{\Q_n})$ the regulator, by $\Tam(A_f/{\Q_n})$ the product of the Tamagawa numbers, and let $\Omega_{A_f/{\Q_n}}=(\Omega_{A_f/{\Q}})^{p^n}$, where $\Omega_{A_f/{\Q}}$ is the real period of $A_f$. We also denote by $\widehat{A}_f$ the dual of $A_f$.
\end{definition}

\begin{conjecture}[(Cyclotomic BSD)] Let $\zeta_{p^n}$ be a primitive $p^n$th root of unity, $d_n^{an}(f)$ the order of vanishing of $L_p(f,\alpha,T)$ at ${T=\zeta_{p^n}-1}$, and $r_n^{an'}(f)$ the order of vanishing of the complex $L$-series $L(f/\Q_n,s):=\prod_{\chi \in \Gal(\Q_n/\Q)} L(f,\chi,s)$ at $s=1$. Then
$$d_n^{an}(f)= \frac{r_n^{an'}(f)-r_{n-1}^{an'}(f)}{p^n-p^{n-1}}\text{ and } \sum_\sigma d_n^{an}(f^\sigma)=d_n.$$
In particular, the order of vanishing $r_n^{an'}$ of $L(A_f/{\Q_n},s):=\prod_\sigma L(f^\sigma/\Q_n,s)$ at $s=1$ is $d_n$.
\end{conjecture}

In view of this conjecture, we put (cf. \cite[Remark 8.5]{manin}):% when $A_f$ is an elliptic curve):
$$\#\Sha^{an}(A_f/{\Q_n}):=\frac{L^{(r_n^{an'})}(A_f/{\Q_n},1)\#A_f^{tor}(\Q_n)\#\widehat{A}_f^{tor}(\Q_n)\sqrt{D(\Q_n)}}{(r_n^{an'})!\Omega_{A_f/{\Q_n}}R(A_f/{\Q_n})\Tam(A_f/{\Q_n})}.$$

Our notation of $d_n^{an}(f)$, which is independent of the choice $\zeta_{p^n}$, is justified as follows:

\begin{lemma}\label{rootsofunitywelldefinedness}We have $$d_n^{an}(f)=\frac{r^{an}_n(f)-r^{an}_{n-1}(f)}{p^n-p^{n-1}}.$$
\end{lemma}

We postpone the proof until after Lemma \ref{ordersofvanishing}. \rm

\remark It is not clear (at least to the author) how to relate the leading Taylor coefficient of $L_p(f,\alpha,T)$ at $T=\zeta_{p^n}-1$ to the size of the \v{S}afarevi\v{c}-Tate groups, even when $A_f$ is an elliptic curve (For a relative version, see \cite[\S 9.5, Conjecture 4]{msd}).

\section{The Mordell-Weil rank in the cyclotomic direction} \rm
We now give an upper bound for $r_\infty^{an}(f)$.
When $f$ is ordinary at $p$, we have the estimate $\lambda\geq r_\infty^{an}(f)$, where $\lambda$ is the $\lambda$-invariant of $L_p(f,\alpha,T)$. This section is devoted to the more complicated supersingular scenario. We give two different upper bounds. To obtain an upper bound on $r_\infty^{an}$, one then simply sums the bounds on $r_\infty^{an}(f^\sigma)$. (\textit{Note that $f^\sigma$ may be ordinary or supersingular at $p$ independently of whether $f$ was!})

\begin{proposition}\label{equiroots} Let $f$ be a weight two modular form and $p$ be a good supersingular prime. If $\zeta$ is a $p^n$th root of unity, then we have
$$\ord_{\zeta-1}L_p(f,\alpha,T)=\ord_{\zeta-1}L_p(f,\beta,T).$$
\end{proposition}
\begin{proof} For $n=0$, this is \cite[Lemma 6.6]{pollack}. 
Thus, let $n>0$. Let us first prove that ${L_p(f,\alpha,\zeta-1)=0}$ if and only if $L_p(f,\beta,\zeta-1)=0$: Observation \ref{theobservation} allows us to conclude that
$$\begin{array}{ll}\left(L_p(f,\alpha,\zeta-1),L_p(f,\beta,\zeta-1)\right)&=\overrightarrow{L}_p(f,\zeta-1)\Log_{\alpha,\beta}(\zeta-1)\\&
=\overrightarrow{L}_p(f,\zeta-1)\smat{*&*\\ *&*}\smat{1&0\\0&\Phi_{p^n}(\zeta)}\smat{-1&-1\\ \beta&\alpha}\smat{\alpha^{-N}&0\\0&\beta^{-N}}\\& =\overrightarrow{L}_p(f,\zeta-1)\smat{*&*\\ *&*}\smat{-1&-1\\\beta\Phi_{p^n}(\zeta)&\alpha\Phi_{p^n}(\zeta)}\smat{\alpha^{-N}&0\\0&\beta^{-N}}\end{array}$$
for some $2\times2$-matrix $\smat{*&*\\ *&*}$ with entries in $\overline{\Q}$. But $\Phi_{p^n}(\zeta)=0$, so $L_p(f,\alpha,\zeta-1)=0$ implies that we have $\overrightarrow{L}_p(\zeta-1)\smat{*&*\\ *&*}=(0,*)$. Thus, we can conclude that $L_p(f,\beta,\zeta-1)=0$. A symmetric argument shows $L_p(f,\beta,\zeta-1)=0$ implies $L_p(f,\alpha,\zeta-1)=0$.

The rest is induction: Fixing $k\in\N$ and assuming $L_p^{(i)}(f,\alpha,\zeta-1)=L_p^{(i)}(f,\beta,\zeta-1)=0$ for $0\leq i<k$,
$$\left(L_p^{(k)}(f,\alpha,\zeta-1),L_p^{(k)}(f,\beta,\zeta-1)\right)=\overrightarrow{L}_p^{(k)}(f,\zeta-1)\Log_{\alpha,\beta}(\zeta-1)$$ by the product rule. By the above argument, 
$L_p^{(k)}(f,\alpha,\zeta-1)=0$ if and only if $L_p^{(k)}(f,\beta,\zeta-1)=0.$\end{proof}

\begin{corollary}\label{pollack!} Let $a_p=0$. Then we have $r^{an}_\infty(f)\leq \lambda_\sharp+\lambda_\flat$.

\rm This had been proved in the elliptic curve case when $p\equiv 3 \mod 4$ and $a_p=0$ by Pollack. He derived Proposition \ref{equiroots} in this case by a very clever argument involving Gau{\ss} sums (for which $p\equiv 3 \mod 4$ is needed) and the functional equation (which is simple enough for elliptic curves).
\begin{proof}[Proof of Corollary \ref{pollack!}] The proof of \cite[Corollary 6.8]{pollack} now works in the desired generality, since the only hard ingredient was Proposition \ref{equiroots}.
\end{proof}

\definition Given an integer $n$, let $\Xi_n$ be the matrix so that $\Log_{\alpha,\beta}=\CCC_1\cdots\CCC_n\Xi_n$.
\begin{lemma}\label{ordersofvanishing}Fix an integer $n$ and let $m\leq n$. We then have
$$\ord_{\zeta_{p^m}-1}L_p(f,\alpha,T)=j \text{ if and only if } \ord_{\zeta_{p^m}-1}\overrightarrow{L}_p(f,T)\CCC_1\cdots\CCC_n=j.$$
\end{lemma}
\begin{proof}Since ${\det \Xi_n(\zeta_{p^m}-1)\neq0}$, we have by induction on $i$ and Proposition \ref{equiroots} that $${L_p^{(i)}(f,\alpha,\zeta_{p^m}-1)=0 } \text{ for $i \leq j-1$ but not $i=j$}$$ is equivalent to $$\overrightarrow{L}_p^{(i)}(f,T)\CCC_1\cdots\CCC_{n}(\zeta_{p^m}-1)=(0,0) \text{ for } i \leq j-1\text{ but not $i=j$, }$$ which is equivalent to $\ord_{\zeta_{p^m}-1}\overrightarrow{L}_p(f,T)\CCC_1\cdots\CCC_{n}=j$.
\end{proof}
\begin{proof}[Proof of Lemma \ref{rootsofunitywelldefinedness}] The entries of the vector $\overrightarrow{L}_p(f,T)\CCC_1\cdots\CCC_n$ are up to units polynomials, so for $m\leq n$, we have $\ord_{\zeta_{p^m}-1}\overrightarrow{L}_p(f,T)\CCC_1\cdots\CCC_n=\ord_{\zeta'_{p^m}-1}\overrightarrow{L}_p(f,T)\CCC_1\cdots\CCC_n$ for any two primitive $p^m$th roots of unity $\zeta_{p^m}$ and $\zeta'_{p^m}$. From Lemma \ref{ordersofvanishing}, $\ord_{\zeta_{p^m}-1}L_p(f,\alpha,T)=\ord_{\zeta'_{p^m}-1}L_p(f,\alpha,T)$.
\end{proof}
\end{corollary}
\notation Given $x\in\Q$, we let $\lfloor x \rfloor$ be the largest integer $\leq x$.
\definition We define the $n$th \textbf{$\sharp/\flat$-Kurihara terms} $q_n^{\sharp/\flat}$ and some auxiliary integers $\nu_{\sharp/\flat}, \widetilde{\nu}_{\sharp/\flat}$.
\[\begin{array}{cllc}q_n^\sharp:=\left\lfloor\frac{p^n}{p+1}\right\rfloor & \text{ if $n$ is odd,} & \text{ and $q_n^\sharp:=q_{n+1}^\sharp$ for even $n$,}\\
q_n^\flat:=\left\lfloor\frac{p^n}{p+1}\right\rfloor & \text{ if $n$ is even,} & \text{ and $q_n^\flat:=q_{n+1}^\flat$ for odd $n$}.\end{array}\]

$$\nu_\sharp:=\text{largest odd integer }n\geq1\text{ so that } \lambda_\sharp\geq p^n-p^{n-1}-q_n^\sharp,$$
$$\nu_\flat:=\text{largest even integer }n\geq2\text{ so that } \lambda_\flat\geq p^n-p^{n-1}-q_n^\flat,$$
$$\widetilde{\nu}_\flat:=\text{largest odd integer }n\geq3\text{ so that } \lambda_\flat\geq p^n-p^{n-1}-pq_{n-1}^\flat-(p-1)^2,$$
$$\widetilde{\nu}_\sharp:=\text{largest even integer }n\geq2\text{ so that } \lambda_\sharp\geq p^n-p^{n-1}-pq_{n-1}^\sharp.$$
In case no such integer exists, we put respectively $\nu_\sharp:=0$, $\nu_\flat:=0,$ $\widetilde{\nu}_\sharp:=0$, but $\widetilde{\nu}_\flat:=1$.
%\end{definition}

Note that explicitly, we have
 \[q_n^\sharp= p^{n-1}-p^{n-2}+p^{n-3}-p^{n-4}+\cdots +p^2-p \text{ for odd $n>1$,}\] \[q_n^\flat=p^{n-1}-p^{n-2}+p^{n-3}-p^{n-4}+\cdots +p-1 \text{ for even $n>0$}.\]
\rm
\begin{convention}We define the $\mu$-invariant of the $0$-function to be $\infty$.
\end{convention}
\begin{theorem}\label{justfornumbers} \begin{enumerate}[$\bullet$]
\item When $|\mu_\sharp-\mu_\flat|\leq v=\ord_p(a_p)$ (e.g. when $a_p=0$), put $\nu=\max(\nu_\sharp,\nu_\flat)$.
We then have $r^{an}_\infty(f)\leq\min(q_\nu^\sharp+\lambda_\sharp,q_\nu^\flat+\lambda_\flat).$
\item When $\mu_\sharp>\mu_\flat+ v,$ put $\nu=\max(\nu_\flat,\widetilde{\nu}_\flat)$. We then have $r^{an}_\infty(f)\leq\min(q_\nu^\flat+\lambda_\flat,q_{\nu-1}^\flat-(p-1)^2+\lambda_\flat)$ when $\nu\neq1$, and $r^{an}_\infty(f)\leq \min(q_1^\flat+\lambda_\flat,q_1^\sharp+\lambda_\sharp)$ when $\nu=1$.
\item When $\mu_\flat>\mu_\sharp+v$, put $\nu=\max(\nu_\sharp,\widetilde{\nu}_\sharp)$. We then have $r^{an}_\infty(f)\leq \min(pq_{\nu-1}^\sharp+\lambda_\sharp, q_\nu^\sharp+\lambda_\sharp).$
\end{enumerate} 
\end{theorem}

\definition For a vector $\overrightarrow{a}=(a_\sharp,a_\flat)\in\Lambda^{\oplus2}$, we define its $\lambda$-invariant as $\lambda(\overrightarrow{a}):=\min(\lambda(a_\sharp),\lambda(a_\flat))$.%\end{definition}

\begin{proof}We handle the first case first. Denote $L_p(f,\alpha,T)$ by $L_\alpha$. In the proof, we justify the two equality signs in the following equation:$$\displaystyle\sum_{\substack{\text{all $p$-power}\\\text{roots of unity }\zeta}} \ord_{\zeta-1} L_\alpha = \sum_{\substack{\zeta\text{ so that }\zeta^{p^n}=1\\ \text{ and } n\leq \nu}} \ord_{\zeta-1}L_\alpha\displaystyle
= \sum_{\substack{\zeta\text{ so that }\zeta^{p^n}=1\\ \text{ and } n\leq \nu}} \ord_{\zeta-1}\overrightarrow{L}_p\CCC_1\cdots\CCC_{\nu}$$ The result then follows on noting that the last term is bounded by $\lambda(\overrightarrow{L}_p\CCC_1\cdots\CCC_{\nu}).$
%$$\begin{cases}\lambda_\sharp+q_n^\sharp \text{ and } \lambda_\flat+q_n^\flat \text{ if $n$ is even,} \\ \lambda_\flat+q_n^\flat \text{ and } \lambda_\sharp+q_n^\sharp \text{ if $n$ is odd.}\end{cases}$$

We justify the first equality sign.
By Proposition \ref{equiroots}, $\ord_{\zeta_{p^n}-1}L_\alpha=\ord_{\zeta_{p^n}-1}L_{\beta}$. Since we have ${\left.\det \Xi_n\right|_{T=\zeta_{p^n}-1}\neq 0,}$ we can say that $\ord_{\zeta_{p^n}-1} L_\alpha =0$ if and only if $\left.\overrightarrow{L}_p\CCC_1\cdots\CCC_n\right|_{T=\zeta_{p^n}-1}\neq(0,0)$. Since $\lambda(\overrightarrow{L}_p\CCC_1\cdots\CCC_n)$ is bounded above by $\lambda_\sharp+q_n^\sharp$ and $\lambda_\flat+q_n^\flat$, we have that \[ p^n-p^{n-1}>\min(\lambda_\sharp+q_n^\sharp,\lambda_\flat+q_n^\flat) \text{ implies} \left.\overrightarrow{L}_p\CCC_1\cdots\CCC_n\right|_{\zeta_{p^n}-1}\neq(0,0).\]
%$$\begin{cases}\lambda_\flat+q_n^\flat\geq p^n-p^{n-1} \text{ if $n$ is even,} \\ \lambda_\sharp+q_n^\sharp\geq p^n-p^{n-1} \text{ if $n$ is odd.} \end{cases}$$
Now $\lambda_\flat +q_n^\flat < p^n-p^{n-1}$ for some even $n$ implies $\lambda_\flat +q_m^\flat< p^m-p^{m-1}$ for any even $m\geq n$. Similarly, $\lambda_\sharp +q_n^\sharp < p^n-p^{n-1}$ for some odd $n$ implies $\lambda_\sharp +q_m^\sharp< p^m-p^{m-1}$ for any odd $m\geq n$. Thus,
$$m>\nu\text{ implies } \ord_{\zeta_{p^m}-1}L_\alpha=0.$$
The second equality sign follows from Lemma \ref{ordersofvanishing} applied to $n=\nu$.

In the other cases, similar arguments hold, with the following caveats: 

In the second case, $\lambda(\overrightarrow{L}_p\CCC_1\cdots\CCC_n)=\begin{cases}\lambda_\flat+q_n^\flat & \text{when $n$ is even,} \\ \lambda_\flat+pq_{n-1}^\flat-(p-1)^2 & \text{when $n$ is odd and $n\neq 1$,}\\ \lambda_\flat+q_1^\flat & \text{ when $n=1,$}\end{cases}$

while in the third case, $\lambda(\overrightarrow{L}_p\CCC_1\cdots\CCC_n)=\begin{cases}\lambda_\sharp+pq_{n-1}^\sharp & \text{when $n$ is even,} \\ \lambda_\sharp+q_n^\sharp & \text{when $n$ is odd.}\\ \end{cases}$

(The asymmetry in the second case comes from $(L_p^\sharp,L_p^\flat)\CCC_1\equiv(-\Phi_pL_p^\flat,L_p^\sharp) \mod{a_p}$.)
\end{proof}

\rm Comparing this bound with the sum of $\lambda$-invariants bound of Corollary \ref{pollack!}, we find that it is in most cases sharper (except when $p=2$, in which case it is \textit{never sharper}. Here, the cases when the bounds match is when there is an odd $\nu$ so that $\lambda_\sharp=q_\nu^\sharp$ and $\lambda_\flat\leq q_\nu^\flat$ or there is an even $\nu$ so that $\lambda_\flat=q_\nu^\flat$ and $\lambda_\sharp\leq q_\nu^\sharp$). When $p$ is odd and $f$ is elliptic modular, this bound is strictly sharper in all known cases (cf. the tables of Perrin-Riou \cite{perrinriou} and Pollack at http://math.bu.edu/people/rpollack/Data/data.html), except when: \begin{enumerate}
                                                                                                                                                                                                                         \item $\lambda_\flat=0$ and $\lambda_\sharp<p-1$,
                                                                                                                                                                                                                         
\item $p=3$, and $(\lambda_\sharp,\lambda_\flat)\in \{(0,6),(1,5),(1,6),(2,4),(2,5),(2,6),(12,2),(13,x) \text{ with $x\leq5$} \}$                                                                                                                                                                                                                         
                                                                                                                                                                                                                        \end{enumerate}

The following corollary gives a bound that is in the spirit of the bound in the ordinary case:

\begin{corollary}
 Assume $\lambda_\sharp<p-1$ and $\lambda_\flat<p^2-p-p+1=(p-1)^2$. Then $r_\infty^{an}(f)\leq \min(\lambda_\sharp,\lambda_\flat)$ for $\mu_\flat\leq\mu_\sharp+v$, while $r_\infty^{an}(f)\leq \lambda_\sharp$ when $\mu_\flat > \mu_\sharp+v$.
\end{corollary}

\proof Indeed, we have $\nu_\sharp=\nu_\flat=\widetilde{\nu}_\sharp=\widetilde{\nu}_\flat-1=0$ in this case.

\example When $p$ is odd and $\lambda_\sharp=\lambda_\flat=1$, we have $r_\infty^{an}=r_\infty^{an}(f)\leq 1$, cf. \cite[Proposition 7.17]{perrinriou} for the elliptic curve case. This case is very common numerically.

\rm We thank Robert Pollack for pointing out the following example in which the sum of the $\lambda$-invariants is not a bound for $r^{an}_\infty(f)$ as in Corollary \ref{pollack!}. Our proposition explains the bound:
\example Consider E37A. For the prime $3$, we have $a_3=-3$, and at this prime $3$, we have $\lambda_\sharp=1,$ while $\lambda_\flat=5$, and $r_\infty^{an}=r_\infty^{an}(f)=7$. In this case $\nu_\sharp=0$ and $\nu_\flat=2$. Thus, the bound for $r_\infty^{an}$ is $\min(q_{2}^\flat+5, q_2^\sharp+1)=\min(3-1+5,3^2-3+1)=7$. Note that $r_\infty^{an}=7>\lambda_\sharp+\lambda_\flat=6.$

\section{The special value of the $L$-function of $f$ in the cyclotomic direction}\rm
The purpose of this section is to prove a special value formula for modular forms of weight two in the cyclotomic direction that estimates  $\Sha(A_f/\Q_n)[p^\infty]$. We encounter an unexpected phenomenon when $v=\ord_p(a_p)<{1\over2}$. %For simplicity and motivation, we first give an explicit version of the formula in the case of elliptic curves in which $v\geq1$ in \ref{subsec:first}. After that, we state it for general modular forms of weight two, before finally giving the proof.

\definition Put
$$\CCC_i(a,1+T):=\links a & 1 \\ -\epsilon(p)\Phi_{p^i}(1+T) & 0 \rechts.$$We now put $\H_a^i(1+T):=\CCC_1(a,1+T)\cdots\CCC_{i}(a,1+T).$
%\end{definition}

\definition\label{kuriharatermsformodularform} Given an element $a$ in the closed unit disc of $\C_p$, let ${v:=\ord_p(a)\geq 0}$. When $v>0$, let $k\in \Z^{\geq 1}$ be the smallest positive integer so that $v\geq{p^{-k}\over 2}$.  We now let $v_m=v_m(a)$ be the upper left entry in the valuation matrix of $\H_a^m(\zeta_{p^{k+2}}-1)$.%=\H_a^m(\zeta_{p^n}^{p^{n-k-2}}-1)$.

Given further an integer $n$, we now define two functions $q_n^*(v,v_2)$ for $*\in\{\sharp,\flat\}$  so that they are \textit{continuous} in $v\in[0,\infty]$ and in $v_2\in[2v,\infty]$.

When $\infty>v>0$, we put  $\delta:=\min(v_2-2v,(p-1)p^{-k-2}).$ Note that $\delta=0$ when $v\neq{p^{-k}\over 2}.$  We define

$$q_n^\sharp(v,v_2):=\begin{cases}(p^n-p^{n-1})kv+\left\lfloor\frac{p^{n-k}}{p+1}\right\rfloor &\text{ when $n\not\equiv k\mod(2)$}\\(p^n-p^{n-1})\left((k-1)v+\delta\right)+\left\lfloor\frac{p^{n+1-k}}{p+1}\right\rfloor&\text{ when $n\equiv k\mod(2)$},\end{cases}$$
$$q_n^\flat(v,v_2):=\begin{cases}(p^n-p^{n-1})\left((k-1)v+\delta\right)+p\left\lfloor\frac{p^{n-k}}{p+1}\right\rfloor+p-1&\text{ when $n\not\equiv k\mod(2)$}\\(p^n-p^{n-1})kv +p\left\lfloor\frac{p^{n-1-k}}{p+1}\right\rfloor+p-1&\text{ when $n\equiv k\mod(2)$}.\end{cases}$$

Note that the tail terms $\left\lfloor\frac{p^{n-k}}{p+1}\right\rfloor$ (resp. $\left\lfloor\frac{p^{n+1-k}}{p+1}\right\rfloor$) appearing in $q_n^\sharp(v,v_2)$ are equal to $q_{n-k}^\sharp$. For $n>k$, those for $q_n^\flat(v,v_2)$, i.e. $p\left\lfloor\frac{p^{n-k}}{p+1}\right\rfloor+p-1$ and $p\left\lfloor\frac{p^{n-1-k}}{p+1}\right\rfloor+p-1$, are both $q_{n-k}^\flat$.
For $v=\infty$, we define $$ q_n^*(\infty,v_2):=\lim_{v\rightarrow \infty}q_n^*(v,v_2).$$
Finally, for $v=0$, we similarly put
$$q_n^*(0,v_2):=\lim_{v\rightarrow 0}q_n^*(v,v_2)=\begin{cases}0&\text{ when $*=\sharp$}\\p-1&\text{ when $*=\flat$.}\end{cases}$$
(We use this seemingly strange adherence to the symbol $v_2$ simply for uniform notation.)

%\end{definition}
\begin{definition}\label{sporadic} The \textbf{sporadic case} (for $v$ and $v_2$) occurs if $v=0$ and $\mu_\sharp=\mu_\flat$ and $\lambda_\sharp=\lambda_\flat+p-1$,
{or if $v=\frac{p^{-k}}{2} $ and $ v_2=2v(1+p^{-1}-p^{-2})$ and}$$\begin{cases}n\not\equiv k \mod{2} \text{ and }\begin{cases}\text{$\mu_\sharp-\mu_\flat>v-\frac{2v}{p^3+p^2}$}\text{ or }\\\text{$\mu_\sharp-\mu_\flat=v-\frac{2v}{p^3+p^2}$ and $\lambda_\sharp>\lambda_\flat$, or}\end{cases} \\ n\equiv k \mod{2} \text{ and } \begin{cases}\text{$\mu_\sharp-\mu_\flat<\frac{2v}{p^3+p^2}-v$}\text{ or }\\\text{$\mu_\sharp-\mu_\flat=\frac{2v}{p^3+p^2}-v$ and $\lambda_\sharp\leq\lambda_\flat$.}\end{cases} \end{cases}$$
\end{definition}

\begin{definition}[(Modesty Algorithm)]\label{modesty}
Given $a$ in the closed unit disc, an integer $n$, integers $\lambda_\sharp$ and $\lambda_\flat$, and rational numbers $\mu_\sharp$ and $\mu_\flat$, choose $*\in\{\sharp,\flat\}$ via
$$*=\begin{cases}\sharp \text{ if $(p^n-p^{n-1})\mu_\sharp+\lambda_\sharp+q_n^\sharp(v,v_2)<(p^n-p^{n-1})\mu_\flat+\lambda_\flat+q_n^\flat(v,v_2)$}\\\flat \text{ if $(p^n-p^{n-1})\mu_\flat+\lambda_\flat+q_n^\flat(v,v_2)<(p^n-p^{n-1})\mu_\sharp+\lambda_\sharp+q_n^\sharp(v,v_2)$.}\end{cases}$$
 \end{definition}

\begin{theorem}\label{specialvalues} Let $f$ be a modular form of weight two which is a normalized eigenform for all $T_n$ with eigenvalue $a_n$ and $p$ a good prime. Let $v=v(a_p)$ and $v_2=v_2(a_p)$ via Definition \ref{kuriharatermsformodularform}. For a character  $\chi$ of $\Z_p^\times$ with order $p^n$, denote by $\tau(\chi)$ the Gau{\ss} \text{ }sum. Let $n$ be large enough so that $\ord_p(L_p^{\sharp/\flat}(f,\zeta_{p^n}-1))=\mu_{\sharp/\flat}+\frac{\lambda_{\sharp/\flat}}{p^n-p^{n-1}}$, and suppose $n>k$ when $v>0$, and suppose we are not in the sporadic case. Then
$$\ord_p\left(\tau(\chi)\frac{L(f,\chi^{-1},1)}{\Omega_f}\right)=\mu_*+\frac{1}{p^n-p^{n-1}}\left(\lambda_*+q_n^*(v,v_2)\right),$$
and $*\in\{\sharp,\flat\}$ is chosen according to the Modesty Algorithm \ref{modesty}. \end{theorem}

\rm See Figure $1$ in the introduction for an illustration of this theorem.

\begin{proof}\rm Let $p$ be odd, since the other case is similar. Letting $\chi(\gamma)=\zeta_{p^n}$, the interpolation property implies
$$L_p(f,\alpha,\zeta_{p^n}-1)=\frac{1}{\alpha^{n+1}}\frac{p^{n+1}}{\tau(\chi^{-1})}\frac{L(f,\chi^{-1},1)}{\Omega_f}.$$ %Since ${\tau(\chi)\tau(\chi^{-1})=\pm p^{n+1}}$, we have $$\tau(\chi)\frac{L(f,\chi^{-1},1)}{\Omega_f}=\pm\alpha^{n+1}L_p(f,\alpha,\zeta_{p^n}-1).$$
Now ${\alpha^{n+1}L_p(f,\alpha,\zeta_{p^n}-1)}$
 has the desired $p$-adic valuation by Proposition \ref{specialvaluesproposition} below and Theorem \ref{maintheorem}.   \end{proof}
 
 For $\sigma\in\Gal(\overline{\Q}/\Q)$, let $\mu_{\sharp/\flat}^\sigma$ and $\lambda_{\sharp/\flat}^\sigma$ be the $\mu$- and $\lambda$-invariants of $L_p^{\sharp/\flat}(f^\sigma,T)$, and let $v^\sigma=v(a_p^\sigma)$ and $v_2^\sigma=v_2(a_p^\sigma)$. For $v^\sigma=0$, put $q_n^\natural(v^\sigma,v_2^\sigma)=0$ and let $\mu_\natural^\sigma$ and $\lambda_\natural^\sigma$ be the $\mu$- and $\lambda$-invariants of $L_p(f^\sigma,\alpha,T)$.
\begin{corollary}\label{shatheorem} Let $p^{e_n}:=\Sha^{an}(A_f/\Q_n)[p^\infty]$. Suppose we are not in the sporadic case for any pair $v^\sigma,v_2^\sigma$ with $v^\sigma>0$. Then for $n \gg 0$, $$e_n-e_{
n-1}=\sum_{\sigma\in\G_f}\mu_*^\sigma(p^n-p^{n-1})+\lambda_*^\sigma+q_n^*(v^\sigma,v_2^\sigma)-r_\infty^{an}(f^\sigma),$$ where  $*\in\{\sharp,\flat\}$ is chosen according to the Modesty Algorithm \ref{modesty}, except when $v^\sigma=0$ (and we are in the sporadic case), in which case $*:=\natural$.
\end{corollary}
\begin{proof}
This follows from Theorem \ref{specialvalues} in the same way that \cite[Proposition 6.10]{pollack} follows from \cite[Proposition 6.9 (3)]{pollack}: The idea is to pick $n$ large enough so that $\ord_p(\#A_f(\Q_n))=\ord_p(\#A_f(\Q_{n-1}))$, $L(A_f,\chi,1)\neq0$ for $\chi$ of order $p^n$, and  $\ord_p(\Tam(A_f/\Q_n))=\ord_p(\Tam(A_f/\Q_{n-1})).$
 Noting that $R(A_f/\Q_n)=p^{r_n}R(A_f/\Q_{n-1})$ and by computing $D(\Q_n)$, $$\begin{array}{ll}e_n-e_{n-1}&=\ord_p\left(\prod_{\chi\text{ of order $p^n$}}\frac{L(A_f/\Q,\chi^{-1},1)}{\Omega_{A_f/\Q}}\right)+p^{n-1}(p-1)\cdot\frac{n+1}{2}-r^{an}_\infty\\ &=\ord_p\left(\prod_{\chi\text{ of order $p^n$}}\tau(\chi)\frac{L(A_f/\Q,\chi^{-1},1)}{\Omega_{A_f/\Q}}\right)-r_\infty^{an}\\&=\ord_p\left(\prod_{\chi\text{ of order $p^n$}}\tau(\chi)\prod_{\sigma}\frac{L(f^\sigma,\chi^{-1},1)}{\Omega_{f^\sigma}}\right)-\sum_\sigma r_\infty^{an}(f^\sigma),\end{array}$$ with the $\sigma$ chosen as in the corollary.\end{proof}

\begin{corollary}
If $A_f$ is an elliptic curve, $\ord_p(L(A_f,1)/\Omega_{A_f})=0$, $a_p\not \equiv 1 \mod p$, $p$ is odd, and $p\nmid \Tam(A_f/\Q_n)$, then $e_0=e_1=0=\mu_{\sharp/\flat}=\lambda_{\sharp/\flat}=r_\infty^{an}$ and the above formulas are valid for $n\geq 2$.\end{corollary}
\begin{proof}We can pick $n=0$ by \cite[Proposition 1.2]{kurihara} and the arguments of its proof, invoking \cite[Proposition 3.8]{greenbergitaly} and Theorem \ref{technicality}.
\end{proof}

\definition[The invariants $\mu_\pm$ and $\lambda_\pm$] Perrin-Riou, resp. Greenberg, Iovita, and Pollack defined invariants $\mu_\pm$ and $\lambda_\pm$ as follows.
Let $p$ be a supersingular prime, and let $p$ be odd \footnote{This is an assumption that Perrin-Riou makes. For $p=2$, one could define the $\mu_\pm$ and $\lambda_\pm$ in the same way but switch the signs so that they agree with the Iwasawa invariants of $L_p^\pm$ in the case $a_p=0$.}. Let $(Q_n)_n\in\Lambda_n$ be a queue sequence. Let $\pi$ be a generator of the maximal ideal of $\O$ so that $\pi^m=p$. When $Q_n\neq0$,  we define $\mu'(Q_n)$ to be the unique integer so that $$Q_n\in(\pi)^{\mu'(Q_n)}\Lambda_n-(\pi)^{\mu'(Q_n)+1}\Lambda_n.$$ Further, we let $\lambda(Q_n)$ be the unique integer so that $\pi^{-\mu'(Q_n)}Q_n \mod \pi\in\tilde{I_n}^{\lambda(Q_n)}-\tilde{I_n}^{\lambda(Q_n)+1},$ where $\tilde{I_n}$ is the augmentation ideal of $\O/\pi\O[\Gamma_n]$. Finally, we put $\mu(Q_n):=m\mu'(Q_n)$. Then for even (resp. odd) $n$, $\mu(Q_n)$ stabilizes to a minimum constant value $\mu_+$ (resp. $\mu_-$). When $\mu_+=\mu_-$, put
$$\lambda_+:=\lim_{n\rightarrow\infty}\lambda(Q_{2n})-q_{2n}^\flat \text{ and }\lambda_-:=\lim_{n\rightarrow\infty}\lambda(Q_{2n+1})-q_{2n+1}^\sharp.$$
%\end{definition}

\begin{corollary}\label{invariantsarethesame}When $\mu_\sharp$ and $\lambda_\sharp$ (resp. $\mu_\flat$ and $\lambda_\flat$) appear in the estimates of Theorem \ref{specialvalues}, they are the Iwasawa invariants of $\widehat{L}_p^\sharp$ (resp. of $\widehat{L}_p^\flat$). When $\mu_\sharp=\mu_\flat$, we define $\mu_\pm$ and $\lambda_\pm$ via the queue sequences that gave rise to $L_p^\sharp$ and $L_p^\flat$, and have$$\mu_\sharp=\mu_+,\lambda_\sharp=\lambda_+,\mu_\flat=\mu_-,\text{ and }\lambda_\flat=\lambda_-.$$
\end{corollary}

\begin{proof}The Kurihara terms $q_n^*(v,v_2)$ come from appropriate valuation matrices of $\Log_{\alpha,\beta}$ and $\widehat{\Log}_{\alpha,\beta}$, which are the same. Thus, the Iwasawa invariants of $\widehat{L}_p^{\sharp/\flat}$ and of $L_p^{\sharp/\flat}$ match. We can calculate the $p$-primary part of the special value in Theorem \ref{specialvalues} using the appropriate queue sequences \footnote{This has been explicitly done in an unpublished preprint of Greenberg, Iovita, and Pollack.}. Since $\mu_+=\mu_-$, we are a posteriori not in the sporadic case, so that our formulas match. \end{proof}

\rm 
\subsection{Remarks in the elliptic curve case}For the remainder of this subsection, assume $A_f=E$ is an elliptic curve. Then in the supersingular case, Corollary \ref{shatheorem} generalizes \cite[Proposition 6.10]{pollack}, which works under the assumption $a_p=0$. For an algebraic version of this Corollary \ref{shatheorem} for supersingular primes, see \cite[Theorem 10.9]{kobayashi} in the case $a_p=0$ and odd $p$, and \cite[Theorem 3.13]{nextpaper} for any odd supersingular prime.

\begin{remark}These formulas are compatible with Perrin-Riou's formulas in \cite{perrinriou}. Note that she assumes that $p$ is odd, and that $\mu_+=\mu_-$ or $a_p=0$ in \cite[Theorem 6.1(4)]{perrinriou}, cf. also \cite[Theorem 5.1]{nextpaper} . Her invariants match ours by Corollary \ref{invariantsarethesame}. For $p=2$, our results are compatible with \cite[Theorem 0.1 (2)]{kuriharaotsuki} (which determines the structure of the $2$-primary component of $\Sha(A_f/\Q_n)$ under the assumption $a_2=\pm2$ in the elliptic curve case and other conditions, which force the Iwasawa invariants to vanish).\end{remark}

In the ordinary case, the estimate for $ n \gg0$ is
$$e_n-e_{n-1}=(p^n-p^{n-1})\mu+\lambda-r_\infty^{an},$$
where $\mu$ and $\lambda$ are the Iwasawa invariants of $L_p(E,\alpha,T)$. Thus, we obtain

\begin{corollary}\label{lambdacomparisonlemma}In the ordinary case, let $\lambda$ be the $\lambda$-invariant of $L_p(E,\alpha,T)$. Then $$\lambda=\begin{cases}\lambda_\sharp\text{ when } \mu_\sharp<\mu_\flat \text{ or } \mu_\sharp=\mu_\flat \text{ and } \lambda_\sharp<\lambda_\flat+p-1 \\ \lambda_\flat\text{ when }\mu_\flat<\mu_\sharp \text{ or } \mu_\flat=\mu_\sharp \text{ and } \lambda_\flat<\lambda_\sharp+1-p.\end{cases}$$
\end{corollary}

%\begin{conjecture}In the above setting, we have
%$$\lambda=\lambda_\sharp=\lambda_\flat \text{ if $\mu_\sharp=\mu_\flat$ and $\lambda_\sharp=\lambda_\flat+p-1$.}$$
%$\end{conjecture}

 \rm
\subsection{Tools for the proof of Theorem \ref{specialvalues}}
\begin{proposition}\label{specialvaluesproposition}Suppose we have $(L^\sharp(T),L^\flat(T))\in\O[[T]]^{\oplus2}$, where $\O$ is the ring of integers of some finite extension of $\Q_p$. Rewrite $L^\sharp(T):=p^{\mu_\sharp}\times P^\sharp(T) \times U^\sharp(T)$ for a distinguished polynomial $P^\sharp(T)$ with $\lambda$-invariant $\lambda_\sharp$ and a unit $U^\sharp(T)$. Note that $\mu_\sharp\in\Q$.  Rewrite $L^\flat(T)$ similarly to extract $\mu_\flat$ and $\lambda_\flat$. Suppose we are not in the sporadic case. Let $a $ and $k$ be as in Definition \ref{kuriharatermsformodularform}, and $e_n$ the left entry of the $1\times2$ valuation matrix of $$(L^\sharp(\zeta_{p^n}-1),L^\flat(\zeta_{p^n}-1))\H_a^{n-1}(\zeta_{p^n}-1).$$
Then for $n$ large enough so that $n>k$ and $\ord_p(L^{\sharp/\flat}(\zeta_{p^n}-1))=\mu_{\sharp/\flat}+\frac{\lambda_{\sharp/\flat}}{p^n-p^{n-1}}$, we have
$$e_n=\mu_*+\frac{\lambda_*}{p^n-p^{n-1}}+\frac{q_n^*(v,v_2)}{p^n-p^{n-1}},$$
where $*\in\{\sharp,\flat\}$ is chosen according to the modesty algorithm.
\end{proposition}
\begin{proof} From Lemma \ref{ordinarylemma} and Lemma \ref{supersingularlemma} below, it follows that the valuation matrix of the above expression is a product (of valuation matrices) of the form $$\left[\mu_\sharp+\frac{\lambda_\sharp}{p^n-p^{n-1}},\mu_\flat+\frac{\lambda_\flat}{p^n-p^{n-1}}\right]\left[\begin{array}{cc}\frac{q_n^\sharp(v,v_2)}{p^n-p^{n-1}}&*\\\frac{q_n^\flat(v,v_2)}{p^n-p^{n-1}}&*\end{array}\right],$$ except when $v= \frac{p^{-k}}{2}$ and $v_2 = p^{-k}(1+p^{-1}-p^{-2})$, in which case one of the two entries shown in the right valuation matrix is the actual entry, while the other is a lower estimate, cf. Lemma \ref{supersingularlemma}.  The leading term of $P^{\sharp/\flat}(T)$ dominates by assumption, so the Modesty Algorithm \ref{modesty} chooses the correct subindex.
\end{proof}

\begin{lemma}\label{basiccalculations}When $v>0$ and $n>k+3$, the valuation matrix ${[\H_a^{n-k-2}(\zeta_{p^n}-1)]}$ is $$\begin{cases}
\left[\begin{array}{cc}p^{2-n}+p^{4-n}+p^{6-n}+\cdots +p^{-k-2} & v+p^{2-n}+\cdots+p^{-k-4}\\ v+p^{1-n}+\cdots+p^{-k-3} &p^{1-n}+\cdots+p^{-k-3} \end{array}\right]&\text{ if $n\equiv k \mod(2)$}\\

\left[\begin{array}{cc} v+p^{2-n}+\cdots+p^{-k-3} &p^{2-n}+\cdots+p^{-k-3}\\p^{1-n}+\cdots+p^{-k-2} &v+p^{1-n}+\cdots+p^{-k-4} \end{array}\right]&\text{ if $n\not\equiv k \mod(2)$.}\end{cases}$$
\end{lemma}
\begin{proof}
Multiplication of valuation matrices and induction.
\end{proof}

\begin{lemma}\label{ordinarylemma}With notation as above, assume $v=0$. Then $$\left[\H_a^{n-1}(\zeta_{p^n}-1)\right]=\left[\begin{array}{cc} 0 & 0 \\ p^{1-n} & p^{1-n} \end{array}\right].
$$\end{lemma}
\begin{proof}Multiplication of valuation matrices.
\end{proof}
Given a real number $x$, recall that $``\geq x"$ denotes an unknown quantity greater than or equal to $x$.

\begin{lemma}\label{supersingularlemma}When $v>0$ and $n>k$, we have $(p^n-p^{n-1})[\H_a^{n-1}(\zeta_{p^n}-1)]=$
$$
\left[\begin{array}{cc}q_n^\sharp(v,v_2)&q_n^\sharp(v,v_2)-v\\ q_n^\flat(v,v_2) & q_n^\flat(v,v_2) -v \end{array}\right]\text{unless $v=\frac{p^{-k}}{2}$ and $v_2=2v(1+p^{-1}-p^{-2})$}.$$
When $v=\frac{p^{-k}}{2}$ and $v_2=2v(1+p^{-1}-p^{-2})$, we have $(p^n-p^{n-1})[\H_a^{n-1}(\zeta_{p^n}-1)]=$$$\begin{cases}
\left[\begin{array}{cc}\geq q_n^\sharp(v,v_2)&\geq q_n^\sharp(v,v_2)-v\\ q_n^\flat(v,v_2) & q_n^\flat(v,v_2) -v \end{array}\right]\text{when $n\equiv k \mod(2)$}\\
\left[\begin{array}{cc} q_n^\sharp(v,v_2)& q_n^\sharp(v,v_2)-v\\ \geq q_n^\flat(v,v_2) & \geq q_n^\flat(v,v_2) -v \end{array}\right]\text{when $n\not\equiv k\mod(2)$}.\end{cases}$$
\end{lemma}
\begin{proof}We give the proof for the case $n\equiv k\mod(2)$ and $n\geq k+4$. (The case where ${n\not\equiv k \mod 2}$ and $n\geq 5$ is similar, and the excluded cases are easier variants of these calculations. \footnote{For $n=k+1$, we directly verify $\left[\H_a^k(\zeta_{p^{k+1}}-1)\right]=\mat{kv & (k-1)v \\ (k-1)v+p^{-k} & (k-2)+p^{-k}}$.}) We have $$\H_a^{n-1}(\zeta_{p^n}-1)=\H_a^{n-k-2}(\zeta_{p^n}-1)\H_a^{k+1}(\zeta_{p^n}^{p^{n-k-2}}-1),$$ whose valuation matrix is the product of valuation matrices
$$\left(\left[\begin{array}{cc}\frac{p^{-k}}{p^2-1}&v+\frac{p^{-k-2}}{p^2-1}\\v+\frac{p^{-k-1}}{p^2-1} &\frac{p^{-k-1}}{p^2-1}\end{array} \right]-\left[\begin{array}{cc} \frac{p^{2-n}}{p^2-1}& \frac{p^{2-n}}{p^2-1}\\ \frac{p^{1-n}}{p^2-1}&\frac{p^{1-n}}{p^2-1} \end{array}\right]\right)\left[\begin{array}{cc}v_{k+1} & v_{k} \\ {k}v+p^{-1-k} & (k-1)v + p^{-1-k}\end{array}\right]$$

 by Lemma \ref{basiccalculations}, where the lower entries in the last valuation matrix are calculated by induction just as in Lemma \ref{basiccalculations} above. The first column of $[\H_a^{n-1}(\zeta_{p^n}-1)]$ is
$$\bai \min(v_{k+1},(k+1)v+p^{-1-k}-p^{-k-2}) +\frac{p^{-k}-p^{2-n}}{p^2-1}\\ \min(v_{k+1},(k-1)v+p^{-1-k})+v+\frac{p^{-k-1}-p^{1-n}}{p^2-1},\end{array}\right.$$
as long as the two terms involved in $\min(\;,\;)$ are different.

If $2v> p^{-k}$, we have $v_{k+1}=(k-1)v+p^{-k}$, so the first column of $[\H_a^{n-1}(\zeta_{p^n}-1)]$ is
$$\bai (k-1)v +p^{-k}+\frac{p^{-k}-p^{2-n}}{p^2-1}\\ kv+p^{-1-k}+\frac{p^{-k-1}-p^{1-n}}{p^2-1}.\end{array}\right.$$
The difficult part is the case $2v=p^{-k}$. From the lemma below, we find that  the expression for the lower term is the same as when $2v>p^{-k}$.

 We claim that the upper term is the same as well (i.e. the minimum is $v_{k+1}$) when $v_2<p^{-k}(1+\frac{1}{p}-\frac{1}{p^2})$, while the minimum is the other term when $v_2>p^{-k}(1+\frac{1}{p}-\frac{1}{p^2})$. For $v_2\geq p^{1-k}$ , this follows at once from the below lemma, since $v_{k+1}\geq(k-1)v+p^{1-k}$; so the real difficulty is when $p^{1-k}>v_2\geq p^{-k}$: Here, the lemma below tells us that $v_{k+1}=v_2+(k-1)\frac{p^{-k}}{2}$, from which we obtain our claim. Note that when $v_2= p^{-k}(1+\frac{1}{p}-\frac{1}{p^2})$, we obtain our desired inequality.
\end{proof}
\begin{lemma}In the above situation, let $m\geq 2$. We then have $v_m=(m-2)v+v_2$ when $v_2<p^{1-k}$ and $v_m\geq (m-2)v+p^{1-k}$ if not.
\end{lemma}
\begin{proof}Explicit decomposition of the valuation matrix of $\H_a^k(\zeta_{p^n}^{p^{n-k-1}}-1)$.
\end{proof}

%$$

\begin{description}
\item[Acknowledgments]
\small We are grateful to our advisors, Barry Mazur, Robert Pollack, and Joseph Silverman, for many inspiring conversations and many pieces of advice in the writing of this paper, and for making the author a (better) mathematician. We also thank Christian Wuthrich for a helpful conversation about regulators at the Iwasawa 2012 conference, and Tadashi Ochiai and Daniel Disegni for pointing out inaccuracies in an earlier version of this paper. We thank J. Pottharst for enlightening comments, and %, and Jay Pottharst for several encouraging comments. %We thank Robert Pollack for teaching the author about queue sequences, and for withholding the calculation mentioned at the end of the introduction.
%Our thanks especially go to Barry Mazur and Jonathan (Jay) Pottharst for expressing their belief that Theorem \ref{maintheorem} (in the supersingular case) would be easier to prove than it turned out to be. It wasn't, but the reward for our involved techniques has been the ordinary case!
  Diana Davis and Maxime Bourque for help with typesetting the picture.
\end{description}

\end{document}